\documentclass[oneside,english,reqno]{amsart}
\usepackage{lmodern}
\usepackage[T1]{fontenc}
\usepackage[latin9]{inputenc}
\usepackage{babel}
\usepackage{amsthm}
\usepackage{amssymb}
\usepackage{graphicx}
\PassOptionsToPackage{normalem}{ulem}
\usepackage{ulem}
\usepackage[unicode=true,pdfusetitle,
 bookmarks=true,bookmarksnumbered=false,bookmarksopen=false,
 breaklinks=false,pdfborder={0 0 1},backref=false,colorlinks=false]
 {hyperref}
\usepackage{breakurl}

\makeatletter
\numberwithin{equation}{section}
\numberwithin{figure}{section}
\theoremstyle{plain}
\newtheorem{thm}{\protect\theoremname}[section]
  \theoremstyle{plain}
  \newtheorem{prop}[thm]{\protect\propositionname}
  \theoremstyle{plain}
  \newtheorem{algorithm}[thm]{\protect\algorithmname}
  \theoremstyle{remark}
  \newtheorem{rem}[thm]{\protect\remarkname}
  \theoremstyle{definition}
  \newtheorem{example}[thm]{\protect\examplename}
  \theoremstyle{definition}
  \newtheorem{defn}[thm]{\protect\definitionname}
  \theoremstyle{plain}
  \newtheorem{lem}[thm]{\protect\lemmaname}

\newcommand{\lin}{\mbox{\rm lin}}

\newcommand{\intr}{\mbox{\rm int}}
\newcommand{\aspan}{\mbox{\rm aff-span}}

\title[Set intersection: Supporting hyperplanes, quadratic programming]{Set intersection problems: Supporting hyperplanes and quadratic programming}

\makeatother

  \providecommand{\algorithmname}{Algorithm}
  \providecommand{\definitionname}{Definition}
  \providecommand{\examplename}{Example}
  \providecommand{\lemmaname}{Lemma}
  \providecommand{\propositionname}{Proposition}
  \providecommand{\remarkname}{Remark}
\providecommand{\theoremname}{Theorem}

\begin{document}
\begin{abstract}
We study how the supporting hyperplanes produced by the projection
process can complement the method of alternating projections and its
variants for the convex set intersection problem. For the problem
of finding the closest point in the intersection of closed convex
sets, we propose an algorithm that, like Dykstra's algorithm, converges
strongly in a Hilbert space. Moreover, this algorithm converges in
finitely many iterations when the closed convex sets are cones in
$\mathbb{R}^{n}$ satisfying an alignment condition. Next, we propose
modifications of the alternating projection algorithm, and prove its
convergence. The algorithm converges superlinearly in $\mathbb{R}^{n}$
under some nice conditions. Under a conical condition, the convergence
can be finite. Lastly, we discuss the case where the intersection
of the sets is empty.
\end{abstract}

\author{C.H. Jeffrey Pang}

\curraddr{Department of Mathematics\\ 
National University of Singapore\\ 
Block S17 05-10\\ 
10 Lower Kent Ridge Road\\ 
Singapore 119076 }

\email{matpchj@nus.edu.sg}

\date{\today{}}



\maketitle
\tableofcontents{}

\section{Introduction}

For finitely many closed convex sets $K_{1},\dots,K_{r}$ in a Hilbert
space $X$, the \emph{Set Intersection Problem }(SIP) is stated as:
\begin{equation}
\mbox{Find }x\in K:=\bigcap_{i=1}^{r}K_{i}\mbox{, where }K\neq\emptyset.\label{eq:SIP}
\end{equation}
 One assumption on the sets $K_{i}$ is that projecting a point in
$X$ onto each $K_{i}$ is a relatively easy problem. 

A popular method of solving the SIP is the \emph{Method of Alternating
Projections }(MAP), where one iteratively projects a point through
the sets $K_{i}$ to find a point in $K$. Another problem related
to the SIP is the \emph{Best Approximation Problem }(BAP): Find the
closest point to $x_{0}$ in $K$, that is, 
\begin{eqnarray}
 & \underset{x\in X}{\min} & \|x-x_{0}\|\label{eq:Proj-pblm}\\
 & \mbox{s.t. } & x\in K:=\bigcap_{i=1}^{r}K_{i}.\nonumber 
\end{eqnarray}
for closed convex sets $K_{i}$, $i=1,\dots,r$. One can easily construct
an example in $\mathbb{R}^{2}$ involving a circle and a line such
that the MAP converges to a point in $K$ that is not $P_{K}(x_{0})$.
Fortunately, Dykstra's algorithm \cite{Dykstra83,BD86} reduces the
problem of finding the projection onto $K$ to the problem of projecting
onto $K_{i}$ individually by adding correction vectors after each
iteration. It was rediscovered in \cite{Han88} using mathematical
programming duality. For more on the background and recent developments
of the MAP and its variants, we refer the reader to \cite{BB96_survey,BR09,EsRa11},
as well as \cite[Chapter 9]{Deustch01} and \cite[Subsubsection 4.5.4]{BZ05}.

We quote \cite{Deutsch01_survey}, where it is mentioned that the
MAP has found application in at least ten different areas of mathematics,
which include: (1) solving linear equations; (2) the Dirichlet problem
which has in turn inspired the \textquotedbl{}domain decomposition\textquotedbl{}
industry; (3) probability and statistics; (4) computing Bergman kernels;
(5) approximating multivariate functions by sums of univariate ones;
(6) least change secant updates; (7) multigrid methods; (8) conformal
mapping; (9) image restoration; (10) computed tomography. See also
\cite{Deutsch95} for more information.

One problem of the MAP and Dykstra's algorithm is slow convergence.
A few acceleration methods were explored. The papers \cite{GPR67,GK89,BDHP03}
explored the acceleration of the MAP using a line search in the case
where $K_{i}$ are linear subspaces of $X$. See \cite{Deutsch01_survey}
for a survey. One can easily rewrite the SIP as a \emph{Convex Inequality
Problem }(CIP):
\[
\mbox{Find }x\in\mathbb{R}^{n}\mbox{ satisfying }g(x)\leq0,
\]
where $g:\mathbb{R}^{n}\to\mathbb{R}^{r}$ is such that each $g_{i}:\mathbb{R}^{n}\to\mathbb{R}$,
where $i=1,\dots,r$, is convex: Just set $g_{i}(x)$ to be the distance
from $x$ to $K_{i}$. In the case where $X=\mathbb{R}^{n}$ and each
$g_{i}(\cdot)$ is differentiable with Lipschitz gradient, the papers
\cite{G-P98,G-P01} proved a superlinear convergent algorithm for
the CIP. They make use of the subgradients of $g(\cdot)$ to define
separating hyperplanes to the feasible set, and make use of quadratic
programming to achieve superlinear convergence. Another related work
is \cite{Kiwiel95}, where the interest is on problems where $r$,
the number of closed convex sets $K_{i}$, is large.

We elaborate on the quadratic programming approach. Given $x_{1}\in X$
and the projection $x_{2}=P_{K_{1}}(x_{1})$, provided $x_{2}\neq x_{1}$,
a standard result on supporting hyperplanes gives us $K_{1}\subset\left\{ x\mid\left\langle x_{1}-x_{2},x\right\rangle \leq\left\langle x_{1}-x_{2},x_{2}\right\rangle \right\} $.
The aim of this work is make use of the supporting hyperplanes generated
in the projection process to accelerate the convergence to a point
in $K$. A relaxation of \eqref{eq:Proj-pblm} is 
\begin{eqnarray}
 & \underset{x\in X}{\min} & \|x-x_{0}\|^{2}\label{eq:quad-prog}\\
 & \mbox{s.t.} & \left\langle a_{i},x\right\rangle \leq b_{i}\mbox{ for }i=1,\dots,k,\nonumber 
\end{eqnarray}
where each constraint $\left\langle a_{i},x\right\rangle \leq b_{i}$
corresponds to a supporting hyperplane obtained by the projection
operation onto one of the sets $K_{i}$, where $1\leq i\leq r$. Let
$S=\mbox{span}\{a_{i}\mid i=1,\dots,k\}$, and let $V=\{v_{1},\dots,v_{k^{\prime}}\}$
be a set of orthonormal vectors spanning $S$, where $k^{\prime}\leq k$.
We can write $a_{i}=\sum_{j=1}^{k^{\prime}}\alpha_{i,j}v_{j}$, $x_{0}=y_{0}+z_{0}$
and $x=\sum_{j=1}^{k^{\prime}}\lambda_{j}v_{j}+z$, where $y_{0}\in S$
and $z_{0},z\in S^{\perp}$. Then \eqref{eq:quad-prog} can be rewritten
as 
\begin{eqnarray}
 & \underset{\lambda\in\mathbb{R}^{k^{\prime}},z\in S^{\perp}}{\min} & \left\Vert \sum_{j=1}^{k^{\prime}}\lambda_{j}v_{j}-y_{0}\right\Vert ^{2}+\|z_{0}-z\|^{2}\label{eq:quad-prog-2}\\
 & \mbox{s.t.} & \left\langle \sum_{j=1}^{k^{\prime}}\alpha_{i,j}v_{j},\sum_{j=1}^{k^{\prime}}\lambda_{j}v_{j}\right\rangle \leq b_{i}\mbox{ for }i=1,\dots,k.\nonumber 
\end{eqnarray}
Therefore \eqref{eq:quad-prog} can be easily solved using convex
quadratic programming, especially when $k$ and $k^{\prime}$ are
small. (See for example \cite[Chapter 16]{NW06}.) 

The quadratic programming formulation \eqref{eq:quad-prog} gathers
information from the supporting hyperplanes to many of the closed
convex sets $K_{i}$, and so is a good approximation to \eqref{eq:Proj-pblm};
the intersection of the halfspaces defined by the supporting hyperplanes
can produce a set that is a better approximation of $K$ than each
$K_{i}$ taken singly. Hence there is good reason to believe that
\eqref{eq:quad-prog} can achieve better convergence than simple variants
of the MAP. As Figure \ref{fig:alt-proj-compare} illustrates, the
supporting hyperplanes can provide a good outer estimate of the intersection
$K_{i}$. Furthermore, as more constraints are added in the quadratic
programming formulation \eqref{eq:quad-prog}, it is possible to use
warm starts from previous iterations to accelerate convergence. In
this paper, we shall only pursue the idea of supplementing the MAP
with supporting hyperplanes and quadratic programming, but not on
the details of the quadratic programming subproblem.

\begin{figure}[h]
\includegraphics[scale=0.4]{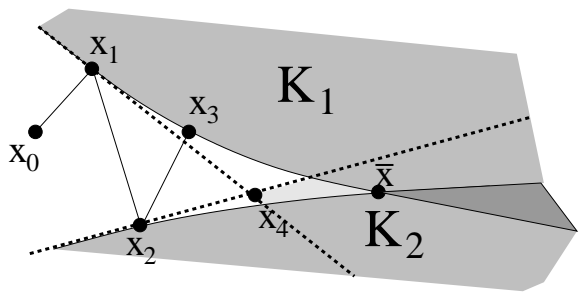}

\caption{\label{fig:alt-proj-compare}The method of alternating projections
on two convex sets $K_{1}$ and $K_{2}$ in $\mathbb{R}^{2}$ with
starting iterate $x_{0}$ arrives at $x_{3}$ in three iterations.
But the point $x_{4}$ generated by the cutting planes of $K_{1}$
and $K_{2}$ at $x_{1}$ and $x_{2}$ respectively is much closer
to the point $\bar{x}$, especially when the boundary of $K_{1}$
and $K_{2}$ have fewer second order effects and when the angle between
the boundary of $K_{1}$ and $K_{2}$ is small. On the other hand,
the point $x_{3}$ is ruled out by the supporting hyperplane of $K_{1}$
passing through $x_{2}$.}
\end{figure}

We remark that the idea using supporting hyperplanes to approximate
the set $K$ was also considered in \cite{BCRZ03}, but their motivation
was to make use of hyperplanes to simplify the projections onto the
sets $K_{i}$ rather than accelerating convergence.

\subsection{Contributions of this paper}

In this paper, we prove theoretical properties of the alternating
projection method supplemented with the insight on supporting hyperplanes.
Sections \ref{sec:closest-pt} to \ref{sec:Infeasibility} are mostly
independent of each other.

First, we propose Algorithm \ref{alg:closest-pt-proj} for the Best
Approximation Problem \eqref{eq:Proj-pblm} in Section \ref{sec:closest-pt}.
We prove norm convergence, and the finite convergence of Algorithm
\ref{alg:closest-pt-proj} with \eqref{eq:mass-proj} when $K_{l}\subset\mathbb{R}^{n}$
have a local conic structure and satisfy a normal condition. We also
show that the normal condition cannot be dropped.

In Section \ref{sec:mod-MAP}, we propose modifications of the alternating
projection algorithm for the Set Intersection Problem \eqref{eq:SIP},
and prove their convergence. We also prove the superlinear convergence
of a modified alternating projection algorithm in $\mathbb{R}^{2}$.

In Section \ref{sec:Error-estimates}, we prove the most striking
result of this paper, which is the superlinear convergence of an algorithm
for the Set Intersection Problem \eqref{eq:SIP} in $\mathbb{R}^{n}$
under reasonable conditions. The convergence can be finite if there
is a local conic structure at the limit point. The proofs of superlinear
convergence are quite different from the proof in Section \ref{sec:mod-MAP}.

Lastly, in Section \ref{sec:Infeasibility}, we discuss the behavior
of Algorithm \ref{alg:closest-pt-proj} in the case when the intersection
of the closed convex sets is empty.

\subsection{Notation}

We shall let $\mathbb{B}_{r}(x)$ be the closed ball with radius $r$
and center $x$. The projection operation onto a set $C$ is denoted
by $P_{C}(\cdot)$. We also make use of standard constructs in convex
analysis.

\section{Some useful results}

In this section, we recall or prove some useful results that will
be useful in two or more of the sections later. The reader may wish
to skip this section and come back to refer to the results as needed.

The result below shows that separating hyperplanes near a point in
a convex set behave well.
\begin{thm}
\label{thm:radiality}(Supporting hyperplane near a point) Suppose
$C\subset\mathbb{R}^{n}$ is convex, and let $\bar{x}\in C$. Then
for any $\epsilon>0$, there is a $\delta>0$ such that for any point
$x\in[\mathbb{B}_{\delta}(\bar{x})\cap C]\backslash\{\bar{x}\}$ and
supporting hyperplane $A$ of $C$ with unit normal $v\in N_{C}(x)$
at the point $x$, we have $\frac{d(\bar{x},A)}{\|x-\bar{x}\|}\leq\epsilon$.

Since $d(\bar{x},A)=-\left\langle v,\bar{x}-x\right\rangle $, the
conclusion of this result can be replaced by $0\leq\frac{-\left\langle v,\bar{x}-x\right\rangle }{\|\bar{x}-x\|}\leq\epsilon$
instead.\end{thm}
\begin{proof}
We refer to Figure \ref{fig:proof-angle}. For the given $\epsilon_{1}>0$,
there is some $\delta>0$ such that if $x\in\mathbb{B}_{\delta}(\bar{x})\cap C$
and $v\in N_{C}(x)$ is a unit vector, then there is some unit vector
$\bar{v}\in N_{C}(\bar{x})$ such that $\|v-\bar{v}\|<\epsilon_{1}$.
This means that the angle between $v$ and $\bar{v}$ is at most $2\sin^{-1}(\epsilon_{1}/2)$. 

One can easily check that $x-\bar{x}$ is not a multiple of $\bar{v}$.
Consider the two dimensional affine space that contains the vector
$\bar{v}$ and the points $x$ and $\bar{x}$, and project the point
$x+v$ onto this affine space. Let this projection be $x+v^{\prime}$.
It is easy to check that the angle between $v^{\prime}$ and $\bar{v}$,
marked as $\alpha$ in Figure \ref{fig:proof-angle}, is bounded from
above by $2\sin^{-1}(\epsilon_{1}/2)$. (The lines with arrows at
both ends passing through $x$ and $\bar{x}$ respectively represent
the intersection of supporting hyperplanes with the two dimensional
affine space.)

The angle $\theta$ in Figure \ref{fig:proof-angle} is an upper bound
on the angle between $x-\bar{x}$ and the supporting hyperplane $A$,
and is easily checked to satisfy $\theta\leq\alpha$. We thus have
\[
\frac{d(\bar{x},A)}{\|x-\bar{x}\|}\leq\sin\theta\leq\sin\alpha\leq\sin\big(2\sin^{-1}(\epsilon_{1}/2)\big).
\]
So for a given $\epsilon>0$, if $\epsilon_{1}$ were chosen to be
such that $\sin\big(2\sin^{-1}(\epsilon_{1}/2)\big)<\epsilon$, then
we are done.
\end{proof}
\begin{figure}[h]
\includegraphics[scale=0.4]{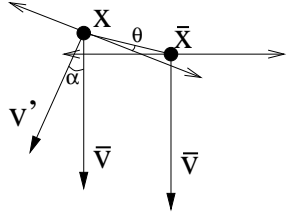}

\caption{\label{fig:proof-angle}Diagram in the proof of Theorem \ref{thm:radiality}. }
\end{figure}

Next, we recall Moreau's Theorem, and remark on how it will be used.
For a convex cone $C$, we denote its negative polar cone by $C^{-}$. 
\begin{thm}
(Moreau's Decomposition Theorem) Suppose $C\subset\mathbb{R}^{n}$
is a closed convex cone. Then for any $x\in\mathbb{R}^{n}$, we can
write $x=P_{C}(x)+P_{C^{-}}(x)$, and moreover, $\left\langle P_{C}(x),P_{C^{-}}(x)\right\rangle =0$.
\end{thm}
The following result will be used in Theorems  \ref{thm:Finite-conv-cones}
and \ref{thm:Superlinear-conv}.
\begin{prop}
\label{rem:Use-Moreau}(Projection onto cones) Suppose $C\subset\mathbb{R}^{n}$
is a closed convex cone. Then the supporting hyperplane formed by
projecting a point $y$ onto $C$ would contain the origin.\end{prop}
\begin{proof}
By Moreau's Theorem, the projection $P_{C}(y)$ satisfies 
\[
y=P_{C}(y)+P_{C^{-}}(y)\mbox{ and }\left\langle P_{C}(y),P_{C^{-}}(y)\right\rangle =0.
\]
The supporting hyperplane produced by projecting $y$ onto $C$ would
be 
\[
\left\{ x\mid\left\langle x,y-P_{C}(y)\right\rangle \leq\left\langle P_{C}(y),y-P_{C}(y)\right\rangle \right\} ,
\]
which equals $\left\{ x\mid\left\langle x,P_{C^{-}}(y)\right\rangle \leq0\right\} $.
It is clear that the origin is in the supporting hyperplane.
\end{proof}

\section{\label{sec:closest-pt}Convergence for the Best Approximation Problem}

In this section, we discuss algorithms for the Best Approximation
Problem \eqref{eq:Proj-pblm}. We describe Algorithm \ref{alg:closest-pt-proj},
and show strong convergence to the closest point in the intersection
of the closed convex sets (Theorem \ref{thm:Strong-conv-result}).
Furthermore, in the finite dimensional case where the sets have a
local conic structure, Algorithm \ref{alg:closest-pt-proj} with \eqref{eq:mass-proj}
converges in finitely many iterations (Theorem \ref{thm:Finite-conv-cones})
under a normal condition \eqref{eq:normal-condn}. We give an example
to show that the condition \eqref{eq:normal-condn} cannot be dropped.

For each $n\in\mathbb{N}$, let $[n]$ denote ``$n$ mod $r$'';
that is, 
\[
[n]:=\{1,2,\dots,r\}\cap\{n-kr\mid k=0,1,2,\dots\}.
\]

We present our algorithm for this section.
\begin{algorithm}
\label{alg:closest-pt-proj}(Best approximation) For a point $x_{0}$
and closed convex sets $K_{l}$, $l=1,2,\dots,r$, of a Hilbert space
$X$, find the closest point to $x_{0}$ in $K:=\cap_{l=1}^{r}K_{l}$. 

\textbf{Step 0: }Let $i=1$.

\textbf{Step 1:} Choose $J_{i}\subset\{1,\dots,r\}$. Some examples
of $J_{i}$ are\begin{subequations}
\begin{eqnarray}
J_{i} & = & \{[i]\},\label{eq:cyclic-proj}\\
\mbox{and }J_{i} & = & \{1,\dots,r\}.\label{eq:mass-proj}
\end{eqnarray}
\end{subequations}For $j\in J_{i}$, define $x_{i}^{(j)}\in X$,
$a_{i}^{(j)}\in X$ and $b_{i}^{(j)}\in\mathbb{R}$ by 
\begin{eqnarray*}
x_{i}^{(j)} & = & P_{K_{j}}(x_{i-1}),\\
a_{i}^{(j)} & = & x_{i-1}-x_{i}^{(j)},\\
\mbox{and }b_{i}^{(j)} & = & \left\langle a_{i}^{(j)},x_{i}^{(j)}\right\rangle .
\end{eqnarray*}
Define the set $F_{i}\subset X$ by
\begin{equation}
F_{i}:=\left\{ x\mid\left\langle a_{l}^{(j)},x\right\rangle \leq b_{l}^{(j)}\mbox{ for all }l=1,\dots,i\mbox{ and }j\in J_{l}\right\} .\label{eq:dykstra-all-planes}
\end{equation}
Let $x_{i}=P_{F_{i}}(x_{0})$.

\textbf{Step 2:} Set $i\leftarrow i+1$, and go back to step 1. 
\end{algorithm}
When $J_{i}$ is chosen using \eqref{eq:cyclic-proj} and $x_{i}^{([i])}\in K_{[i]}$
so that $x_{i}^{([i])}=P_{K_{[i]}}(x_{i-1})$, then $a_{i}^{([i])}=0$
and $b_{i}^{([i])}=0$, and the algorithm stalls for one step. These
values of $a_{i}^{([i])}$ and $b_{i}^{([i])}$ are still valid, though
any implementation should treat this case separately. When the algorithm
stalls for $r$ iterations in a row, then we have found the closest
point from $x_{0}$ to $\bigcap_{l=1}^{r}K_{l}$.
\begin{rem}
\label{rem:proj-to-2nd-order}(Projecting to sets with greater second
order behavior) In Step 1 of Algorithm \ref{alg:closest-pt-proj},
one needs to choose $J_{i}$. When the size of the quadratic programs
are small and easy to solve, it would be ideal to choose $J_{i}$
so that $|J_{i}|=1$. The cyclic choice in \eqref{eq:cyclic-proj}
is a natural choice. But as remarked in Figure \ref{fig:alt-proj-compare},
one factor in our strategy is the second order behavior of the sets
$K_{l}$. Another strategy is to record the distances in the most
recent projections to the set $K_{l}$, and choose $J_{i}$ to contain
the index where the highest distance was recorded. In the case where
one of the sets $K_{l}$ is a subspace (and has fewer second order
effects), the computations would be focused on the other sets. However,
one may want to ensure that all sets are projected to every once in
a while so that Algorithm \ref{alg:closest-pt-proj} is not fooled
in regions where the boundary is locally but not globally affine.
Possible strategies are: 
\begin{eqnarray*}
 &  & \mbox{There exists }p\mbox{ such that for all }\bar{i},\mbox{ }\bigcup_{i=\bar{i}}^{\bar{i}+p}J_{i}=\{1,\dots,r\},\\
 & \mbox{or} & \mbox{For each }l=1,\dots,r\mbox{, there are infinitely many }J_{i}\mbox{ containing }l.
\end{eqnarray*}

\end{rem}
The following theorem addresses the convergence of Algorithm \ref{alg:closest-pt-proj}.
This theorem can be compared to the Boyle-Dykstra Theorem \cite{BD86},
which establishes the convergence of Dykstra's algorithm \cite{Dykstra83}. 
\begin{thm}
\label{thm:Strong-conv-result}(Strong convergence of Algorithm \ref{alg:closest-pt-proj})
For any starting point $x_{0}$, the sequences $\{x_{i}\}$ produced
by Algorithm \ref{alg:closest-pt-proj} using \eqref{eq:cyclic-proj}
or \eqref{eq:mass-proj} converge strongly to $P_{K}(x_{0})$.\end{thm}
\begin{proof}
We shall only prove the result for the choice \eqref{eq:cyclic-proj},
since the proof for \eqref{eq:mass-proj} is similar. By considering
a translation if necessary, we can let $x_{0}$ be $0$. We can also
assume that $0\notin K$. The iterates $x_{i}$ satisfy $\|x_{i}\|\leq d(0,K)$,
so $\{x_{i}\}$ has a weak cluster point $z$. Since $x_{i}$ are
the closest point from $0$ to $F_{i}$, and 
\begin{equation}
F_{i+1}\subset F_{i}\mbox{ for all }i,\label{eq:F-supset}
\end{equation}
we see that $\|x_{i}\|$ is an increasing sequence, so $M:=\lim_{i\to\infty}\|x_{i}\|$
exists.

\textbf{\uline{Step 1: $z$ is actually a strong cluster point.}}\textbf{
}It is clear that $\lim_{i\to\infty}\|x_{i}\|\geq\|z\|$. We only
need to prove that 
\begin{equation}
\|z\|=\lim_{i\to\infty}\|x_{i}\|,\label{eq:limits-of-norms}
\end{equation}
 since this condition together with the weak convergence of the subsequence
of $x_{i}$ implies the strong convergence to $z$. Suppose instead
that $\lim_{i\to\infty}\|x_{i}\|>\|z\|$. Then there is some $k$
such that $\|x_{k}\|>\|z\|$. By \eqref{eq:F-supset}, we have, for
all $i>k$, 
\[
\left\langle x_{k},x_{i}\right\rangle \geq\left\langle x_{k},x_{k}\right\rangle =\|x_{k}\|^{2}>\|x_{k}\|\|z\|\geq\left\langle x_{k},z\right\rangle ,
\]
contradicting $z$ being a weak cluster point of $\{x_{i}\}$. Therefore
$z$ is a strong cluster point of $\{x_{i}\}$. 

\textbf{\uline{}}\textbf{\uline{Step 2: Any $z$ is in $K$.}}
Suppose on the contrary that $z\notin K$. Then there is some $l^{*}$
such that $z\notin K_{l^{*}}$, or $P_{K_{l^{*}}}(z)\neq z$. Algorithm
\ref{alg:closest-pt-proj} generates a hyperplane that separates $z$
from $K_{l^{*}}$. The halfspace $\{x\mid\left\langle a_{z},x\right\rangle \leq b_{z}\}$
separates $z$ and $K$, where for $y\in X$, $a_{y}$ and $b_{y}$
are defined by 
\[
a_{y}=y-P_{K_{l^{*}}}(y),b_{z}=\left\langle y-P_{K_{l^{*}}}(y),P_{K_{l^{*}}}(y)\right\rangle .
\]
The distance $D$ from $0$ to the intersection of halfspaces 
\[
\left\{ x\mid\left\langle -z,x\right\rangle \leq-\|z\|^{2}\mbox{ and }\left\langle a_{z},x\right\rangle \leq b_{z}\right\} 
\]
would satisfy $D>\|z\|$.

Next, the variables $a_{y}$ and $b_{y}$ depend continuously on the
parameter $y$, at $y=z$. This means that if $x_{i}$ is sufficiently
close to $z$ and $[i]=l^{*}$, then the distance $d(0,x_{i+1})$
would be sufficiently close to $D$. This would mean that $\|x_{i}\|>\|z\|$
for $i$ large enough, which is a contradiction to \eqref{eq:limits-of-norms}.
Thus $z\in K$ as needed.

\textbf{\uline{Step 3:}}\uline{ }\textbf{\uline{$z=P_{K}(x_{0})$.}}
To see this, observe that $z\in K$ implies that $d(0,K)\leq\|z\|$.
The fact that $\|z\|=\lim_{i\to\infty}\|x_{i}\|$ from step 1 gives
$d(0,K)=\|z\|$.

Thus we are done.\end{proof}
\begin{rem}
(Reducing number of supporting hyperplanes in defining $F_{i}$) In
the proof of Theorem \ref{thm:Strong-conv-result}, step 1 relies
on the fact that $F_{i+1}\subset F_{i}$ for all $i$ in the choice
of $F_{i}$ in \eqref{eq:dykstra-all-planes}. If $X=\mathbb{R}^{n}$,
then step 1 of the proof would be unnecessary, but the sequence $\{\|x_{i}-x_{0}\|\}$
needs to be increasing in order for step 2 to work. This can be enforced
by adding the hyperplane with normal $(x_{0}-x_{i-1})$ through $x_{i-1}$
in constructing $F_{i}$. To ensure that each quadratic programming
problem that needs to be solved is easy, the polyhedron $F_{i}$ can
be chosen such that the number of inequalities that define $F_{i}$
is small. One can take only the active hyperplanes in solving the
projection problem $x_{i}=P_{F_{i}}(x_{0})$, or by aggregating some
of the active hyperplanes to one active hyperplane when building up
the polyhedron $F_{i}$. 
\end{rem}
In the case where $K_{j}$ are cones, each supporting hyperplane contains
the point $0$. This means that there are no second order effects
near $0$, which in turn gives fast convergence in $\mathbb{R}^{n}$.
\begin{thm}
\label{thm:Finite-conv-cones}(Finite convergence for conical problems
in $\mathbb{R}^{n}$) For Algorithm \ref{alg:closest-pt-proj} with
\eqref{eq:mass-proj}, suppose that $X=\mathbb{R}^{n}$. Convergence
is guaranteed by Theorem \ref{thm:Strong-conv-result}. Suppose $P_{K}(x_{0})=\bar{x}$,
and $K_{j}$ are such that  
\begin{equation}
\mbox{For some }\bar{\epsilon}>0\mbox{, }[K_{j}-\bar{x}]\cap\bar{\epsilon}\mathbb{B}=T_{K_{j}}(\bar{x})\cap\bar{\epsilon}\mathbb{B}\mbox{ for }j=1,\dots,r,\label{eq:cone-condn}
\end{equation}
and 
\begin{equation}
x_{0}-\bar{x}\in\intr\big(N_{K}(\bar{x})\big).\label{eq:normal-condn}
\end{equation}
Then Algorithm \ref{alg:closest-pt-proj} with \eqref{eq:mass-proj}
converges to $\bar{x}$ in finitely many iterations.\end{thm}
\begin{proof}
We can assume $\bar{x}=0$. Suppose on the contrary that the convergence
to $0$ requires infinitely many iterations. We seek a contradiction.
Let $\{x_{i}\}$ be the sequence generated by Algorithm \ref{alg:closest-pt-proj},
and let $\{\tilde{x}_{i}\}$ be a subsequence such that $\|\tilde{x}_{i}-0\|<\bar{\epsilon}$
for all $i$, and $\lim_{i\to\infty}\frac{\tilde{x}_{i}}{\|\tilde{x}_{i}\|}$
exists, say $\tilde{x}$. 

\textbf{\uline{Step 1: $\tilde{x}$ lies in $T_{K}(0)$.}} Suppose
on the contrary that $\tilde{x}\notin T_{K}(0)$. Then $\tilde{x}\notin K_{j}$
for some $j$. Assume without loss of generality that $j=1$. Let
$P_{T_{K_{1}}(0)}(\tilde{x})=z$, and $\tilde{x}-z\in N_{K_{1}}(0)$.
Let $v_{i}=\frac{\tilde{x}_{i}}{\|\tilde{x}_{i}\|}-P_{T_{K_{1}}(0)}\left(\frac{\tilde{x}_{i}}{\|\tilde{x}_{i}\|}\right)$
and $v=\tilde{x}-z$. By the continuity of the projection, we must
have $v_{i}\to v$. Since the hyperplane $\{x\mid\left\langle x,v\right\rangle =\left\langle z,v\right\rangle \}$
separates $\tilde{x}$ from $K_{1}$ and $\left\langle z,v\right\rangle =0$
by Moreau's Theorem, we have $\left\langle \tilde{x},v\right\rangle >0$.

Let $y$ be any point in $\bar{\epsilon}\mathbb{B}$, taking into
account \eqref{eq:cone-condn} and $\bar{x}=0$. By Moreau's Theorem
(See Proposition \ref{rem:Use-Moreau}), the supporting hyperplane
produced by projecting $y$ onto $K_{1}$ contains $0$ on its boundary.
By the design of Algorithm \ref{alg:closest-pt-proj}, we must have
$\left\langle \tilde{x}_{i+1},v_{i}\right\rangle \leq0$, which gives
\[
\left\langle \frac{\tilde{x}_{i+1}}{\|\tilde{x}_{i+1}\|},v_{i}\right\rangle \leq0.
\]
Taking limits, we get $\left\langle \tilde{x},v\right\rangle \leq0$,
which contradicts $\left\langle \tilde{x},v\right\rangle >0$ earlier.

\textbf{\uline{Step 2: $\tilde{x}$ cannot lie in $T_{K}(0)$.}}\textbf{
}Suppose otherwise. Then the condition \eqref{eq:normal-condn} implies
that if $\tilde{x}\in K$ and $i$ is large enough, then $d(x_{0},0)<d(x_{0},\tilde{x}_{i})$,
contradicting that $d(x_{0},0)>d(x_{0},\tilde{x}_{i})$ in the choice
of $\tilde{x}_{i}$.

The statements proved in Steps 1 and 2 are clearly contradictory,
which ends our proof.
\end{proof}
In view of the above result, we would expect Algorithm \ref{alg:closest-pt-proj}
(especially with \eqref{eq:mass-proj}) to converge quickly to the
closest point under condition \eqref{eq:normal-condn}.

The number of iterations needed before convergence depends on, among
other things, the $\bar{\epsilon}$. In the case where $K_{j}$ are
cones and \eqref{eq:normal-condn} does not hold, step 2 in the proof
of Theorem \ref{thm:Finite-conv-cones} may fail, and there may be
no finite convergence. We give an example.
\begin{example}
\label{exa:No-finite-conv}(No finite convergence) Consider $X=\mathbb{R}^{3}$.
Consider the rays 
\[
r_{1}=\mathbb{R}_{+}(1,-1,-1)\mbox{ and }r_{2}=\mathbb{R}_{+}(-1,-1,-1).
\]
For a vector $v$, let $\theta_{1}$ be the angle $r_{1}$ makes with
$v$, and let $\theta_{2}$ be similarly defined. Let $K_{1}$ and
$K_{2}$ be the ice cream cones defined by 
\[
K_{i}=\{v\mid\cos(\theta_{i})\leq1/\sqrt{3}\}\mbox{ for }i=1,2.
\]
Let $x_{0}=(0,0,1)$. A few consequences are immediate.
\begin{enumerate}
\item The ray $\mathbb{R}_{+}(0,-1,0)$ is on the boundaries of $K_{1}$,
$K_{2}$ and $K:=K_{1}\cap K_{2}$.
\item $P_{K}(x_{0})=0$.
\item There is only one unit vector in $N_{K_{1}}\big((0,-1,0)\big)$, say
$u$. Let the subspace $S$ be $\{x\mid(1,0,0)^{T}x=0\}$. Then 
\[
[\mathbb{R}_{+}\{u\}+\{(0,-1,0)\}]\cap S=(0,-1,0).
\]
A similar statement holds when $K_{1}$ is replaced by $K_{2}$. 
\end{enumerate}
We now show that the convergence of Algorithm \ref{alg:closest-pt-proj}
with \eqref{eq:mass-proj} is infinite. By symmetry, the iterates
$x_{i}$ lie in $S$. If Algorithm \ref{alg:closest-pt-proj} with
\eqref{eq:mass-proj} converges in finitely many iterations, then
property (3) would imply that the next to last iterate is of the form
$(0,-\alpha,0)$, where $\alpha>0$, and that cannot happen. In the
case where $x_{0}=(0,\epsilon,1)$, where $\epsilon>0$ is arbitrarily
small, we will still get finite convergence to $0$, but the number
of iterations needed will be arbitrarily large as $\epsilon\searrow0$.
\end{example}

\section{\label{sec:mod-MAP}Convergence for the Set Intersection Problem}

In this section, we analyze a modified alternating projection algorithm
(Algorithm \ref{alg:local-modified-MAP}). The global convergence
of this algorithm is proved in Theorem \ref{thm:basic-conv}. The
insight on supporting hyperplanes allows us to obtain local superlinear
convergence in $\mathbb{R}^{2}$, although Algorithm \ref{alg:local-modified-MAP}
in its current form does not converge superlinearly in $\mathbb{R}^{3}$
(Example \ref{exa:No-superlin}). A locally superlinearly convergence
algorithm will be presented and analyzes in Section \ref{sec:Error-estimates}
using very different methods.

We shall analyze the following algorithm. 
\begin{algorithm}
\label{alg:local-modified-MAP}(Modified MAP) For a point $x_{0}$
and closed convex sets $K_{1}$ and $K_{2}$ of a Hilbert space $X$,
find a point in $K:=K_{1}\cap K_{2}$. 

\textbf{Step 0}: Set $i=1$.

\textbf{Step 1:} Choose $J_{i}\subset\{1,2\}$. Some examples are
\begin{subequations} 
\begin{eqnarray}
J_{i} & = & \{[i]\},\label{eq:MAP-1-pt}\\
\mbox{and }J_{i} & = & \{1,2\}.\label{eq:MAP-2-pt}
\end{eqnarray}
\end{subequations}

\textbf{Step 2:} For $j\in J_{i}$, define $x_{i}^{(j)}\in X$, $a_{i}^{(j)}\in X$
and $b_{i}^{(j)}\in X$ by 
\begin{eqnarray*}
x_{i}^{(j)} & = & P_{K_{j}}(x_{i-1}),\\
a_{i}^{(j)} & = & x_{i-1}-x_{i}^{(j)},\\
\mbox{and }b_{i}^{(j)} & = & \left\langle a_{i}^{(j)},x_{i}^{(j)}\right\rangle .
\end{eqnarray*}
Define the set $F_{i}\subset X$ by
\[
F_{i}:=\begin{cases}
\left\{ x\mid\left\langle a_{l}^{([l])},x\right\rangle \leq b_{l}^{([l])}\mbox{ for }l=i-1,i\right\}  & \mbox{ if }J_{i}=\{[i]\}\mbox{ and }i>1,\\
\left\{ x\mid\left\langle a_{1}^{(1)},x\right\rangle \leq b_{1}^{(1)}\right\}  & \mbox{ if }J_{i}=\{[i]\}\mbox{ and }i=1,\\
\left\{ x\mid\left\langle a_{i}^{(j)},x\right\rangle \leq b_{i}^{(j)}\mbox{ for }j=1,2\right\}  & \mbox{ if }J_{i}=\{1,2\}.
\end{cases}
\]

Let $x_{i}=P_{F_{i}}(x_{i}^{([i])})$. 

\textbf{Step 3: }Set $i\leftarrow i+1$, and go back to step 1. 
\end{algorithm}
As mentioned in Remark \ref{rem:proj-to-2nd-order}, there are good
reasons for choosing $J_{i}$ to be such that $|J_{i}|=1$ but not
cyclic, but the construction of $F_{i}$ has to be amended accordingly.
It may turn out that $x_{i}$ could be in $K_{1}$ already, so $P_{K_{1}}(x_{i})$
will not give a new supporting hyperplane. In this case, we can just
use the supporting hyperplane obtained from previous iterations. When
$J_{i}=\{1,2\}$, we can check that $x_{i}$ lies in the plane containing
$x_{i-1}$, $x_{i}^{(1)}$ and $x_{i}^{(2)}$, and that $x_{i}=P_{F_{i}}(x_{i}^{([i])})=P_{F_{i}}(x_{i-1})$.
We shall prove the superlinear convergence of this case in $\mathbb{R}^{2}$
in Theorem  \ref{thm:Superlin}.

We now recall some results on Fej\'{e}r monotonicity to prove convergence
of Algorithm \ref{alg:local-modified-MAP}. We take our results from
\cite[Theorem 4.5.10 and Lemma 4.5.8]{BZ05}.
\begin{defn}
(Fej\'{e}r monotone sequence) Let $X$ be a Hilbert space, let $C\subset X$
be a closed convex set and let $\{x_{i}\}$ be a sequence in $X$.
We say that $\{x_{i}\}$ is\emph{ Fej\'{e}r monotone with respect
to $C$} if 
\[
\|x_{i+1}-c\|\leq\|x_{i}-c\|\mbox{ for all }c\in C\mbox{ and }i=1,2,\dots
\]
\end{defn}
\begin{thm}
\label{thm:Fejer-ppty}(Properties of Fej\'{e}r monotonicity) Let
$X$ be a Hilbert space, let $C\subset X$ be a closed convex set
and let $\{x_{i}\}$ be a Fej\'{e}r monotone sequence with respect
to $C$. Then 
\begin{enumerate}
\item $\{x_{i}\}$ is bounded and $d(C,x_{i+1})\leq d(C,x_{i})$.
\item $\{x_{i}\}$ has at most one weak cluster point in $C$.
\item If $\intr(C)\neq\emptyset$, then $\{x_{i}\}$ converges in norm.
\end{enumerate}
\end{thm}
\begin{lem}
\label{lem:attractive-ppty-of-projection}(Attractive property of
projection) Let $X$ be a Hilbert space and let $C\subset X$ be a
closed convex set. Then $P_{C}:X\to X$ is \emph{1-attracting with
respect to $C$}: For every $x\notin C$ and $y\in C$, we have 
\[
\|P_{C}(x)-x\|^{2}\leq\|x-y\|^{2}-\|P_{C}(x)-y\|^{2}.
\]

\end{lem}
We now prove the convergence of Algorithm \ref{alg:local-modified-MAP}.
\begin{thm}
\label{thm:basic-conv}(Convergence of Algorithm \ref{alg:local-modified-MAP})
Suppose $K_{1}$ and $K_{2}$ are closed convex sets in a Hilbert
space $X$ such that $K:=K_{1}\cap K_{2}\neq\emptyset$. Then the
iterates in Algorithm \ref{alg:local-modified-MAP} with either \eqref{eq:MAP-1-pt}
or \eqref{eq:MAP-2-pt} are such that $x_{i}$ converges weakly to
some $z$. The convergence is strong if either $\intr(K)\neq\emptyset$
or $X=\mathbb{R}^{n}$.\end{thm}
\begin{proof}
We shall first prove convergence when $J_{i}$ is chosen by \eqref{eq:MAP-1-pt}.
We note that Algorithm \ref{alg:local-modified-MAP} can be easily
extended to the case of $r>2$ closed convex sets, and the corresponding
extension of this result will still be true. 

The sequences $\{x_{2i+1}^{(1)}\}_{i}$ and $\{x_{2i}^{(2)}\}_{i}$
lie in $K_{1}$ and $K_{2}$ respectively. Construct the sequence
$\{\tilde{x}_{i}\}$ such that 
\[
\tilde{x}_{i}=\begin{cases}
x_{j} & \mbox{ if }i=2j\\
x_{2j+1}^{(1)} & \mbox{ if }i=4j+1\\
x_{2j+2}^{(2)} & \mbox{ if }i=4j+3.
\end{cases}
\]
Note that $\{\tilde{x}_{i}\}$ lines up the points in $\{x_{i}\}$
and $\{x_{i}^{([i])}\}$ in the order in which they were produced
in Algorithm \ref{alg:local-modified-MAP}.

\textbf{\uline{Step 1: $\{\tilde{x}_{i}\}$ is Fej\'{e}r monotone
with respect to $K$.}} Since $K\subset K_{1}$, $K\subset K_{2}$
and $K\subset F_{i}$ for all $i$, the projections $P_{K_{1}}$,
$P_{K_{2}}$ and $P_{F_{i}}$ are nonexpansive. So 
\begin{eqnarray*}
 &  & \|x_{i}^{([i])}-y\|=\|P_{K_{[i]}}(x_{i-1})-y\|\leq\|x_{i-1}-y\|\\
 & \mbox{and } & \|x_{i}-y\|=\|P_{F_{i}}(x_{i}^{(i)})-y\|\leq\|x_{i}^{([i])}-y\|\mbox{ for all }y\in K\mbox{ and }i\geq1.
\end{eqnarray*}
This means that $\{\tilde{x}_{i}\}$ is a Fej\'{e}r monotone sequence
with respect to $K$. 

\textbf{\uline{Step 2: $\{\tilde{x}_{i}\}$ is asymptotically regular,
i.e.,}} 
\[
\lim_{i\to\infty}\|\tilde{x}_{i}-\tilde{x}_{i+1}\|=0.
\]
Fix any $\bar{y}\in K$. Applying Lemma \ref{lem:attractive-ppty-of-projection},
we get 
\begin{eqnarray*}
 &  & \|x_{i}^{([i])}-x_{i-1}\|^{2}=\|P_{K_{[i]}}(x_{i-1})-x_{i-1}\|^{2}\leq\|x_{i-1}-\bar{y}\|^{2}-\|x_{i}^{([i])}-\bar{y}\|^{2}\\
 & \mbox{ and } & \|x_{i}-x_{i}^{([i])}\|^{2}\leq\|x_{i}^{([i])}-\bar{y}\|^{2}-\|x_{i}-\bar{y}\|^{2}\mbox{ for all }i\geq1.
\end{eqnarray*}
This tells us that $\|\tilde{x}_{i}-\tilde{x}_{i-1}\|^{2}\leq\|\tilde{x}_{i-1}-\bar{y}\|^{2}-\|\tilde{x}_{i}-\bar{y}\|^{2}$
for all $i\geq1$. Since $\{\tilde{x}_{i}\}$ is Fej\'{e}r monotone
with respect to $K$, $\|\tilde{x}_{i}-\bar{y}\|^{2}$ is a decreasing
sequence. We thus have the asymptotic regularity\emph{ }of $\{\tilde{x}_{i}\}$. 

\textbf{\uline{Step 3: Wrapping up. }}By Theorem \ref{thm:Fejer-ppty}(1),
the sequence $\{\tilde{x}_{i}\}$ is bounded. So $\{\tilde{x}_{i}\}$
has a convergent subsequence, say $\{\tilde{x}_{i_{k}}\}_{k}$. By
the asymptotic regularity of $\{\tilde{x}_{i}\}$, the sequence $\{\tilde{x}_{i_{k}+1}\}_{k}$
has the same limit as $\{\tilde{x}_{i_{k}}\}_{k}$, so we can take
a different subsequence if necessary and assume that infinitely many
of the $i_{k}$ are odd. We can choose yet another subsequence of
$\{\tilde{x}_{i_{k}}\}$ if necessary so that all terms are in either
$K_{1}$, or all terms are in $K_{2}$. For the sake of argument,
assume that all terms lie in $K_{1}$. So the weak limit of $\{\tilde{x}_{i_{k}}\}_{k}$,
say $x$, lies in $K_{1}$. By the asymptotic regularity of $\{\tilde{x}_{i_{k}}\}_{k}$
and considering $\{\tilde{x}_{i_{k}+2}\}_{k}$, we see that $x\in K_{2}$.
So the weak cluster point must lie in $K$. By Theorem \ref{thm:Fejer-ppty}(2),
we conclude that $\{\tilde{x}_{i}\}$ converges to a point in $K$.
The last sentence of the result follows from Theorem \ref{thm:Fejer-ppty}(3).
\textbf{ }

For the case of using \eqref{eq:MAP-2-pt}, the steps are very similar,
so we only give an outline: One proves that the sequences $\{x_{i}\}$
and $\{x_{i}^{(j)}\}$ are Fej\'{e}r monotone with respect to $K$
for $j=1,2$. Next, the sequence $x_{0},x_{1}^{(j)},x_{1},x_{1}^{(j)},x_{2},\dots$
is asymptotically regular, which implies that the sequences $\{x_{i}\}$
and $\{x_{i}^{(j)}\}$ have the same weak cluster points. Since $j$
is arbitrary, the weak cluster points must lie in $K$, and by Theorem
\ref{thm:Fejer-ppty}(2), such a weak cluster point is unique.
\end{proof}
The problem of whether the MAP can converge strongly in a Hilbert
space has only been recently resolved to be negative in \cite{Hundal04},
so it remains to be seen how Theorem \ref{thm:basic-conv} can be
strengthened.

We now move on to the fast local convergence of Algorithm \ref{alg:local-modified-MAP}.
Even though the result below is only valid for $\mathbb{R}^{2}$ and
a result establishing superlinear convergence for $\mathbb{R}^{n}$
is presented in Section \ref{sec:Error-estimates}, Theorem \ref{thm:Superlin}
has value because the proof is simpler than and very different from
the proof in Section \ref{sec:Error-estimates}, and the assumptions
needed are quite different.
\begin{thm}
\label{thm:Superlin}(Superlinear convergence in $\mathbb{R}^{2}$)
Suppose $K_{1}$ and $K_{2}$ are closed convex sets in $\mathbb{R}^{2}$
such that 
\begin{enumerate}
\item Algorithm \ref{alg:local-modified-MAP} with \eqref{eq:MAP-2-pt}
converges to a point $\bar{x}$ such that $\partial N_{K_{1}}(\bar{x})\cap\partial[-N_{K_{2}}(\bar{x})]=\{0\}$,
and 
\item There is some iterate $x_{i}$ such that $x_{i}\notin\intr(K_{j})$
for $j=1,2$. 
\end{enumerate}
Then the sequence $\{x_{i}\}$ thus produced converges locally superlinearly
to $\bar{x}$.\end{thm}
\begin{proof}
We refer to Figure \ref{fig:local-ana-2}. Let $\alpha_{i}^{(1)}$
be the angle between $x_{i}-x_{i}^{(1)}$ and $\bar{x}-x_{i}^{(1)}$,
and let $\alpha_{i}^{(2)}$ be similarly defined. As $i\to\infty$,
the points $x_{i}^{(1)}$ and $x_{i}^{(2)}$ converge to $\bar{x}$,
so Theorem \ref{thm:radiality} says that the angles $\alpha_{i}^{(1)}$
and $\alpha_{i}^{(2)}$ converge to zero. 
\begin{figure}
\includegraphics[scale=0.6]{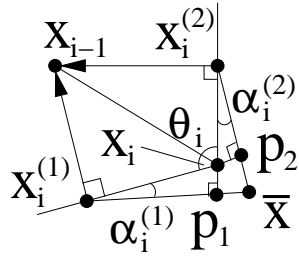}\caption{\label{fig:local-ana-2}Diagram for the proof of Theorem \ref{thm:Superlin}.}
\end{figure}

Let $\theta_{i}$ be the angle between $x_{i}^{(1)}-x_{i}$ and $x_{i}^{(2)}-x_{i}$
as marked. Since $\partial N_{K_{1}}(\bar{x})\cap\partial[-N_{K_{2}}(\bar{x})]=\{0\}$,
the angle $\theta_{i}$ is bounded from below by $\bar{\theta}>0$.
It is also easy to check that if $x_{i}\notin\intr(K_{j})$ for $j=1,2$,
then the same property holds for all $i$ afterward. 

The points $p_{1}$ and $p_{2}$ are obtained by projecting $x_{i}$
onto the line segments $[x_{i}^{(1)},\bar{x}]$ and $[x_{i}^{(2)},\bar{x}]$.
To show that $\{x_{i}\}$ converges superlinearly to $\bar{x}$, it
suffices to show that 
\begin{equation}
\lim_{i\to\infty}\frac{\|x_{i}-\bar{x}\|}{\|x_{i-1}-x_{i}\|}=0,\label{eq:ratio-1}
\end{equation}
since $\|x_{i-1}-\bar{x}\|\geq\|x_{i-1}-x_{i}\|$. Let $L_{i}=\|x_{i-1}-x_{i}\|$. 

By the sine rule, the distance $\|x_{i}^{(1)}-x_{i}\|$ equals $L_{i}\sin\gamma_{i}^{(1)}$,
where $\gamma_{i}^{(1)}$ is some angle in the interval $[0,\pi-\theta_{i}]$.
The distance $\|p_{1}-x_{i}\|$ can be calculated to be bounded above
by $L_{i}\sin\alpha_{i}^{(1)}\sin\gamma_{i}^{(1)}$, while the distance
$\|p_{2}-x_{i}\|$ is easily computed to be bounded from above by
$L_{i}\sin\alpha_{i}^{(2)}\sin\gamma_{i}^{(2)}$, where $\gamma_{i}^{(2)}$
is similarly defined. The distance $\|x_{i}-\bar{x}\|$ is easily
seen to be the diameter of the circumcircle of the cyclic quadrilateral
with vertices $x_{i}$, $\bar{x}$, $p_{1}$ and $p_{2}$. The angle
between $p_{1}-x_{i}$ and $p_{2}-x_{i}$ is easily calculated to
be $\pi-\theta_{i}+\alpha_{i}^{(1)}+\alpha_{i}^{(2)}$. (Note that
$x_{i}^{(2)}$, $x_{i}$ and $p_{1}$ need not be collinear.) The
distance of $\|p_{1}-p_{2}\|$ can be estimated by 
\begin{eqnarray*}
\|p_{1}-p_{2}\| & \leq & \|p_{1}-x_{i}\|+\|p_{2}-x_{i}\|\\
 & \leq & \sin(\min\{\pi/2,\pi-\theta_{i}\})[\sin\alpha_{i}^{(1)}+\sin\alpha_{i}^{(2)}]L_{i}.
\end{eqnarray*}
The value $\|x_{i}-\bar{x}\|$ can be obtained by the sine rule to
be 
\[
\frac{\|p_{1}-p_{2}\|}{\sin(\pi-\theta_{i}+\alpha_{i}^{(1)}+\alpha_{i}^{(2)})},
\]
so we have 
\[
\|x_{i}-\bar{x}\|\leq\frac{\sin(\min\{\pi/2,\pi-\theta_{i}\})}{\sin(\pi-\theta_{i}+\alpha_{i}^{(1)}+\alpha_{i}^{(2)})}L_{i}\left[\sin\alpha_{i}^{(1)}+\sin\alpha_{i}^{(2)}\right].
\]
Thus to prove that \eqref{eq:ratio-1}, it suffices to prove that
\begin{equation}
\lim_{i\to\infty}\frac{\sin(\min\{\pi/2,\pi-\theta_{i}\})}{\sin([\pi-\theta_{i}]+\alpha_{i}^{(1)}+\alpha_{i}^{(2)})}\left[\sin\alpha_{i}^{(1)}+\sin\alpha_{i}^{(2)}\right]=0.\label{eq:limit-zero-est}
\end{equation}
We have shown that $\liminf_{i\to\infty}\theta_{i}\geq\bar{\theta}>0$.
The limit \eqref{eq:limit-zero-est} holds because the limits of $\alpha_{i}^{(1)}$
and $\alpha_{i}^{(2)}$ are zero and $\theta_{i}\in[\bar{\theta},\pi]$
for all $i$. Hence we are done.
\end{proof}
The superlinear convergence in Theorem \ref{thm:Superlin} does not
extend to $\mathbb{R}^{3}$ however, even when $K_{1}$ and $K_{2}$
are linear subspaces.
\begin{example}
\label{exa:No-superlin}(No superlinear convergence in $\mathbb{R}^{3}$
for Algorithm \ref{alg:local-modified-MAP}) We give an example of
subspaces $K_{1}$ and $K_{2}$ in $\mathbb{R}^{3}$ such that $\partial N_{K_{1}}(\bar{x})\cap\partial[-N_{K_{2}}(\bar{x})]=\{0\}$
but there is no superlinear convergence to $\bar{x}$ in Algorithm
\ref{alg:local-modified-MAP} using \eqref{eq:MAP-2-pt} for some
starting point. Consider $K_{1}$ and $K_{2}$ defined by 
\begin{eqnarray*}
K_{1} & = & \mathbb{R}(1,0,1),\\
\mbox{ and }K_{2} & = & \{(x,y,0)\mid x,y\in\mathbb{R}\}.
\end{eqnarray*}
For the starting point $x_{0}=(4,-1,0)$, we compute the iterates
of Algorithm \ref{alg:local-modified-MAP}. We calculate 
\begin{eqnarray}
x_{1}^{(1)} & = & P_{K_{1}}(x_{0})=(2,0,2),\nonumber \\
x_{1}^{(2)} & = & x_{0},\nonumber \\
x_{1} & = & \left(\frac{2}{5},\frac{4}{5},0\right),\nonumber \\
x_{2}^{(1)} & = & P_{K_{1}}(x_{1})=\left(\frac{1}{5},0,\frac{1}{5}\right),\nonumber \\
x_{2}^{(2)} & = & x_{1},\nonumber \\
\mbox{and }x_{2} & = & \left(\frac{16}{85},\frac{-4}{85},0\right)=\frac{4}{85}x_{0}.\label{eq:x2-x0-ratio}
\end{eqnarray}
To verify that $x_{1}$ and $x_{2}$ are the correct iterates, we
can check that $x_{0},$ $x_{1}^{(1)}$, $x_{1}^{(2)}$ and $x_{1}$
lie in the plane $\{x\mid(1,2,0)^{T}x=2\}$, and that $x_{1},$ $x_{2}^{(1)}$,
$x_{2}^{(2)}$ and $x_{2}$ lie in the plane $\{x\mid(4,-1,0)^{T}x=4/5\}$.
Another condition helpful for the verification is that 
\[
\left\langle x_{i}-x_{i}^{(1)},x_{i-1}-x_{i}^{(1)}\right\rangle =0\mbox{ for }i=1,2.
\]

From \eqref{eq:x2-x0-ratio}, we see that the convergence to zero
of Algorithm \ref{alg:local-modified-MAP} using \eqref{eq:MAP-2-pt}
is linear and not superlinear. But the rate of convergence for our
choice of starting iterate is $\frac{4}{85}$ for every four projections,
which is more than twice as fast of the rate of $\frac{1}{4}$ for
every four projections for the usual MAP.

We show that if there were more supporting hyperplanes used in approximating
$K$, then we get finite convergence to zero for this example. The
projection of $x_{i}$ onto $x_{i+1}^{(1)}$ generates the supporting
hyperplanes
\begin{eqnarray*}
 &  & \{x\mid(2,-1,2)x=0\}\mbox{ if }i\mbox{ is even},\\
 & \mbox{ and } & \{x\mid(1,4,-1)x=0\}\mbox{ if }i\mbox{ is odd}.
\end{eqnarray*}
The projection of any point of the form $(t,0,t)$, where $t>0$,
onto the set 
\[
\left\{ x\mid\left(\begin{array}{ccc}
2 & -1 & 2\\
1 & 4 & -1\\
0 & 0 & 1
\end{array}\right)x\leq0\right\} 
\]
is equal to the zero vector, which is the only point in $K$.
\end{example}

\section{\label{sec:Error-estimates}Superlinear convergence for the Set Intersection
Problem }

Our main result in this section is Theorem \ref{thm:Superlinear-conv},
where we prove the superlinear convergence of an algorithm for the
Set Intersection Problem \eqref{eq:SIP} when the normal cones at
the point of intersection are pointed cones satisfying appropriate
alignment conditions. 

We first describe our algorithm for this section. 
\begin{algorithm}
\label{alg:Mass-proj-alg}(Mass projection algorithm) For a starting
iterate $x_{0}$ and closed convex sets $K_{l}\subset\mathbb{R}^{n}$,
where $1\leq l\leq r$, find a point in $K:=\cap_{l=1}^{r}K_{l}$. 

\textbf{Step 0}: Set $i=1$, and let $\bar{p}$ be some positive integer.

\textbf{Step 1:} Choose $J_{i}=\{1,\dots,r\}$.

\textbf{Step 2:} For $j\in J_{i}$, define $x_{i}^{(j)}\in\mathbb{R}^{n}$,
$a_{i}^{(j)}\in\mathbb{R}^{n}$ and $b_{i}^{(j)}\in\mathbb{R}^{n}$
by 
\begin{eqnarray*}
x_{i}^{(j)} & = & P_{K_{j}}(x_{i-1}),\\
a_{i}^{(j)} & = & x_{i-1}-x_{i}^{(j)},\\
\mbox{and }b_{i}^{(j)} & = & \left\langle a_{i}^{(j)},x_{i}^{(j)}\right\rangle .
\end{eqnarray*}
Define the set $\tilde{F}_{i}\subset\mathbb{R}^{n}$ by
\[
\tilde{F}_{i}:=\left\{ x\mid\left\langle a_{l}^{(j)},x\right\rangle \leq b_{l}^{(j)}\mbox{ for }1\leq j\leq r,\max(1,i-\bar{p})\leq l\leq i\right\} .
\]

Let $x_{i}=P_{\tilde{F}_{i}}(x_{i-1})$. 

\textbf{Step 3: }Set $i\leftarrow i+1$, and go back to step 1. 
\end{algorithm}
The modifications in Algorithm \ref{alg:Mass-proj-alg} from Algorithm
\ref{alg:local-modified-MAP} are that we set $X=\mathbb{R}^{n}$,
the number of sets $r$ is arbitrary, and the set $\tilde{F}_{i}$
approximating $K$ is created using more of the previous separating
halfspaces produced earlier.

Algorithm \ref{alg:Mass-proj-alg} produces a sequence $\{x_{i}\}$
Fej\'{e}r monotone with respect to $K$ and converging to a point
$\bar{x}\in K$. The proof is an easy adaptation of that of Theorem
\ref{thm:basic-conv}. 

We recall a well known fact about convex cones.
\begin{prop}
\label{prop:cone-decomp}(Convex cone decomposition) A closed convex
cone $C\subset\mathbb{R}^{n}$ can be written as the direct sum $C=L\oplus[L^{\perp}\cap C]$,
where $L$ is the lineality subspace of $C$ and $L^{\perp}\cap C$
is a pointed convex cone.
\end{prop}
As a consequence of Proposition \ref{prop:cone-decomp}, we have the
following result on the normal cones of convex sets. We denote the
lineality space of a convex set $C$ by $\lin(C)$. The affine space
spanned by $C$ is denoted by $\aspan(C)$.
\begin{prop}
\label{prop:N-C-lineality}(Lineality spaces of normals of convex
sets) Suppose $C\subset\mathbb{R}^{n}$ is a convex set. Then for
any $x\in C$, $[\aspan(C)]^{\perp}=\lin(N_{C}(x))$. In particular,
$\aspan(C)\cap N_{C}(x)$ is a pointed convex cone.\end{prop}
\begin{proof}
$v\in\lin(N_{C}(x))\iff\pm v\in N_{C}(x)\iff\left\langle v,x-c\right\rangle =0$
for all $c\in C$ $\iff v\in[\aspan(C)]^{\perp}$.
\end{proof}
The following result shows that under certain conditions, the directions
from which the iterates converge to the limit must lie inside the
normal cone of $K$ at the limit.
\begin{lem}
\label{lem:Approach-to-x-bar}(Approach of iterates to $\bar{x}$)
For the problem of finding a point $x\in K$, where $K=\cap_{l=1}^{r}K_{l}$
and $K_{l}\subset\mathbb{R}^{n}$ are closed convex sets, suppose
Algorithm \ref{alg:Mass-proj-alg} produces a sequence $\{x_{i}\}$
that converges to a point $\bar{x}\in K$ and is Fej\'{e}r monotone
with respect to $K$. Assume that:
\begin{enumerate}
\item If $\sum_{l=1}^{r}v_{l}=0$ for some $v_{l}\in N_{K_{l}}(\bar{x})$,
then $v_{l}=0$ for all $l=1,\dots,r$.
\end{enumerate}
Then provided none of the $x_{i}$ equals $\bar{x}$, we have 
\begin{equation}
\lim_{i\to\infty}\frac{\|P_{N_{K}(\bar{x})}(x_{i}-\bar{x})\|}{\|x_{i}-\bar{x}\|}=1.\label{eq:proj-span-N-ratio}
\end{equation}
\end{lem}
\begin{proof}
Condition (1) and \cite[Theorem 6.42]{RW98} imply that 
\begin{equation}
N_{K}(\bar{x})=\sum_{l=1}^{r}N_{K_{l}}(\bar{x}).\label{eq:normal-sum-formula}
\end{equation}
 By the way Algorithm \ref{alg:Mass-proj-alg} is designed, the KKT
conditions for the problem of projecting $x_{i-1}$ onto the polyhedron
to obtain $x_{i}$ give 
\[
x_{i}=x_{i-1}-\sum_{l=1}^{r}\sum_{k=\max(1,i-\bar{p})}^{i}[\lambda_{l}^{(i,k)}v_{l}^{k}+w_{l}^{(i,k)}],
\]
where $\lambda_{l}^{(i,k)}v_{l}^{k}+w_{l}^{(i,k)}$ is a multiple
of the vector $a_{k}^{(l)}=x_{k-1}-P_{K_{l}}(x_{k-1})$, $w_{l}^{(i,k)}\in[\aspan(K_{l})]^{\perp}$,
$v_{l}^{k}$ is a unit vector in $\aspan(K_{l})\cap N_{K_{l}}(P_{K_{l}}(x_{k-1}))$,
and $\lambda_{l}^{(i,k)}\geq0$. (The relationship $[\aspan(K_{l})]^{\perp}=\lin(N_{K_{l}}(P_{K_{l}}(x_{k-1})))$
follows from Proposition \ref{prop:N-C-lineality}.) For $j>i$, we
can write $x_{i-1}-x_{j}$ as
\begin{eqnarray*}
x_{i-1}-x_{j} & = & \sum_{s=i}^{j}\sum_{l=1}^{r}\sum_{k=\max(1,s-\bar{p})}^{s}[\lambda_{l}^{(s,k)}v_{l}^{k}+w_{l}^{(s,k)}]\\
 & = & \sum_{l=1}^{r}\left[\left[\sum_{s=i}^{j}\sum_{k=\max(1,s-\bar{p})}^{s}\lambda_{l}^{(s,k)}v_{l}^{k}\right]+\tilde{w}_{l}^{(i-1,j)}\right],
\end{eqnarray*}
where $\tilde{w}_{l}^{(i-1,j)}\in[\aspan(K_{l})]^{\perp}$. Let $\tilde{v}_{l}^{(i,j)}\in\mathbb{R}^{n}$
be the vector 
\begin{equation}
\tilde{v}_{l}^{(i,j)}:=\frac{\sum_{s=i}^{j}\sum_{k=\max(1,s-\bar{p})}^{s}\lambda_{l}^{(s,k)}v_{l}^{k}}{\left\Vert \sum_{s=i}^{j}\sum_{k=\max(1,s-\bar{p})}^{s}\lambda_{l}^{(s,k)}v_{l}^{k}\right\Vert }.\label{eq:v-tilde-i-j-l}
\end{equation}

\textbf{\uline{Claim 1: All cluster points of }}\uline{$\{v_{l}^{k}\}_{k=1}^{\infty}$
}\textbf{\uline{lie in $\aspan(K_{l})\cap N_{K_{l}}(\bar{x})$
for $l=1,\dots,r$.}} This claim is clear from the outer semicontinuity
of the normal cone mapping. 

\textbf{\uline{Claim 2: The infinite sum}} 
\begin{equation}
z_{l,i}:=\sum_{s=i}^{\infty}\sum_{k=\max(1,s-\bar{p})}^{s}\lambda_{l}^{(s,k)}v_{l}^{k}\label{eq:def-z-l-i}
\end{equation}
\textbf{\uline{exists as a limit for $l=1,\dots,r$. Hence $\lim_{j\to\infty}\tilde{v}_{l}^{(i,j)}$
exists.}} 

Suppose on the contrary that $z_{l,i}$ does not exist as a limit
for some $l$, $1\leq l\leq r$. It follows that 
\begin{equation}
\sum_{s=i}^{\infty}\sum_{k=\max(1,s-\bar{p})}^{s}\lambda_{l}^{(s,k)}=\infty,\label{eq:sum-coeff-infty}
\end{equation}
because if the sum in \eqref{eq:sum-coeff-infty} were finite, $z_{l,i}$
would exist as a limit. Note that the cone $\aspan(K_{l})\cap N_{K_{l}}(\bar{x})$
is pointed. Using Claim 1 and Proposition \ref{prop:limit-est-w-pointed-cones}(1),
the subsequence $\{\tilde{v}_{l}^{(i,j)}\}_{j=1}^{\infty}$ has cluster
points in $\aspan(K_{l})\cap N_{K_{l}}(\bar{x})$. Let 
\[
\alpha_{i,j}=\max_{1\leq l\le r}\left\{ \left\Vert \sum_{s=i}^{j}\sum_{k=\max(1,s-\bar{p})}^{s}\lambda_{l}^{(s,k)}v_{l}^{k}\right\Vert ,\|\tilde{w}_{l}^{(i-1,j)}\|\right\} .
\]
We have $\limsup_{j\to\infty}\alpha_{i,j}=\infty$ since \eqref{eq:sum-coeff-infty}
holds for some $l$, and by Proposition \ref{prop:pointed-cone-modulus},
there is a constant $m$ dependent only on $\aspan(K_{l})\cap N_{K_{l}}(\bar{x})$
such that 
\[
\left\Vert \sum_{s=i}^{j}\sum_{k=\max(1,s-\bar{p})}^{s}\lambda_{l}^{(s,k)}v_{l}^{k}\right\Vert \geq m\sum_{s=i}^{j}\sum_{k=\max(1,s-\bar{p})}^{s}\lambda_{l}^{(s,k)}.
\]
Consider the equation 
\begin{equation}
\frac{1}{\alpha_{i,j}}[x_{i-1}-x_{j}]=\sum_{l=1}^{r}\bigg[\underbrace{\frac{1}{\alpha_{i,j}}\left(\sum_{s=i}^{j}\sum_{k=\max(1,s-\bar{p})}^{s}\lambda_{l}^{(s,k)}v_{l}^{k}\right)}_{t_{l,i,j}}+\underbrace{\frac{1}{\alpha_{i,j}}\tilde{w}_{l}^{(i-1,j)}}_{t_{l,i,j}^{\prime}}\bigg].\label{eq:w-K-l}
\end{equation}
It is clear that the LHS converges to zero as $j\to\infty$. We can
choose a subsequence such that the limits $t_{l,i}:=\lim_{j\to\infty}t_{l,i,j}$
and $t_{l,i}^{\prime}:=\lim_{j\to\infty}t_{l,i,j}^{\prime}$, where
$t_{l,i,j}$ and $t_{l,i,j}^{\prime}$ are defined in \eqref{eq:w-K-l},
exist and are not all zero for $1\leq l\leq r$. This would contradict
Condition (1), ending the proof of Claim 2.

In view of Claim 2, define 
\begin{equation}
\tilde{v}_{l}^{(i)}:=\lim_{j\to\infty}\tilde{v}_{l}^{(i,j)}.\label{eq:v-tilde-i-l}
\end{equation}
 Define the matrix $A^{(i,j)}\in\mathbb{R}^{n\times r}$ whose
$l$th column is $\tilde{v}_{l}^{(i,j)}$. We can write 
\begin{equation}
A^{(i,j)}\gamma^{(i,j)}=\sum_{l=1}^{r}\sum_{s=i}^{j}\sum_{k=\max(1,s-\bar{p})}^{s}\lambda_{l}^{(s,k)}v_{l}^{k},\label{eq:A-lambda-i-K}
\end{equation}
where $\gamma^{(i,j)}\in\mathbb{R}^{r}$ is such that $\gamma_{l}^{(i,j)}:=\|\sum_{s=i}^{j}\sum_{k=\max(1,s-\bar{p})}^{s}\lambda_{l}^{(s,k)}v_{l}^{k}\|$
for $l=1,\dots,r$. Let $A^{(i)}:=\lim_{j\to\infty}A^{(i,j)}$ and
$\gamma^{(i)}\in\mathbb{R}^{r}$ be such that 
\[
\gamma_{l}^{(i)}:=\|z_{l,i}\|=\left\Vert \sum_{s=i}^{\infty}\sum_{k=\max(1,s-\bar{p})}^{s}\lambda_{l}^{(s,k)}v_{l}^{k}\right\Vert .
\]
Then 
\[
A^{(i)}\gamma^{(i)}=\sum_{l=1}^{r}z_{l,i}=\sum_{l=1}^{r}\sum_{s=i}^{\infty}\sum_{k=\max(1,s-\bar{p})}^{s}\lambda_{l}^{(s,k)}v_{l}^{k}.
\]
Let
\begin{eqnarray}
\mathcal{A} & := & \{A\in\mathbb{R}^{n\times r}\mid\mbox{The }l\mbox{th column of }A\mbox{ is a }\nonumber \\
 &  & \phantom{\{A\in\mathbb{R}^{n\times r}\mid\mbox{the}}\mbox{ unit vector in }\aspan(K_{l})\cap N_{K_{l}}(\bar{x})\mbox{ for }1\leq l\leq r\},\nonumber \\
L & := & \bigcap_{l=1}^{r}\aspan(K_{l}),\nonumber \\
\mbox{ and }\beta & := & \inf\left\{ \frac{\|P_{L}(A\gamma)\|}{\|\gamma\|}\mid A\in\mathcal{A}\mbox{ and }\gamma\in\mathbb{R}^{r}\backslash\{0\}\mbox{ satisfies }\gamma\geq0\right\} .\label{eq:beta-def}
\end{eqnarray}

\textbf{\uline{Claim 3: $\beta>0$.}} Suppose otherwise. Then there
are sequences of matrices $\tilde{A}^{(i)}\in\mathcal{A}$ and unit
vectors $\tilde{\gamma}^{(i)}\in\mathbb{R}^{r}$ such that $\tilde{\gamma}^{(i)}\geq0$
and $P_{L}(\tilde{A}^{(i)}\tilde{\gamma}^{(i)})\to0$ as $i\nearrow\infty$.
By taking cluster points of $\tilde{A}^{(i)}$ and $\tilde{\gamma}^{(i)}$,
we obtain $P_{L}(\tilde{A}\tilde{\gamma})=0$ for some $\tilde{A}\in\mathcal{A}$
and $\tilde{\gamma}\neq0$, where $\tilde{\gamma}\geq0$. This contradicts
Condition (1), so Claim 3 is proved.

\textbf{\uline{Claim 4:}}\uline{ }\textbf{\uline{$\lim_{i\to\infty}[\inf_{A\in\mathcal{A}}\|A-A^{(i)}\|]=0$.}}
The $l$th column of $A^{(i)}$ is the unit vector $\tilde{v}_{l}^{(i)}$
as defined in \eqref{eq:v-tilde-i-j-l} and \eqref{eq:v-tilde-i-l},
and each $v_{l}^{k}$ lies in $\aspan(K_{l})\cap N_{K_{l}}\big(\mathbb{B}_{\delta}(\bar{x})\big)$,
where $\delta=\|x_{k-1}-\bar{x}\|$. Since $\{x_{i}\}$ converges
to $\bar{x}$ and is Fej\'{e}r monotone, for any $\delta>0$, we
can find $i^{\prime}$ large enough so that $\|x_{i}-\bar{x}\|<\delta$
for all $i>i^{\prime}$. This would mean that for all $\epsilon>0$,
we can find $i$ large enough so that each $v_{l}^{k}$ in the sum
\eqref{eq:v-tilde-i-j-l} satisfies $\|v_{l}^{k}-\bar{v}_{l}^{k}\|<\epsilon$
for some unit vector $\bar{v}_{l}^{k}\in\aspan(K_{l})\cap N_{K_{l}}(\bar{x})$. 

Let 
\[
\hat{v}_{l}^{(i)}:=\frac{\sum_{s=i}^{j}\sum_{k=\max(1,s-\bar{p})}^{s}\lambda_{l}^{(s,k)}\bar{v}_{l}^{k}}{\left\Vert \sum_{s=i}^{j}\sum_{k=\max(1,s-\bar{p})}^{s}\lambda_{l}^{(s,k)}\bar{v}_{l}^{k}\right\Vert }.
\]
Recall that $\tilde{v}_{l}^{(i)}$ is the $l$th column of $A^{(i)}$.
By Proposition \ref{prop:limit-est-w-pointed-cones}(2), there is
a constant $m$ dependent only on $\aspan(K_{l})\cap N_{K_{l}}(\bar{x})$
such that 
\[
\|\tilde{v}_{l}^{(i)}-\hat{v}_{l}^{(i)}\|\leq\epsilon m.
\]
Since $\epsilon\searrow0$ as $i\nearrow\infty$, we can see that
the conclusion to Claim 4 holds.

   Since $\gamma^{(i)}\geq0$. It is clear from the definition
of $\beta$ and Claim 4 that if the $\gamma^{(i)}$'s are nonzero,
then 
\begin{equation}
\liminf_{i\to\infty}\frac{\|P_{L}(A^{(i)}\gamma^{(i)})\|}{\|\gamma^{(i)}\|}\geq\beta.\label{eq:A-gamma-div-gamma}
\end{equation}
To prove that the conclusion \eqref{eq:proj-span-N-ratio} holds,
it suffices to prove that 
\begin{equation}
\lim_{i\to\infty}\left(\inf_{A\in\mathcal{A}}\frac{\|A\gamma^{(i)}-A^{(i)}\gamma^{(i)}\|}{\|x_{i-1}-\bar{x}\|}\right)=0.\label{eq:key-ineq}
\end{equation}
The reason why \eqref{eq:key-ineq} is sufficient is as follows. The
vector $\gamma^{(i)}$ has nonnegative components, and $x_{i-1}-\bar{x}=A^{(i)}\gamma^{(i)}+\sum_{l=1}^{r}\tilde{w}_{l}^{(i)}$
for some $\tilde{w}_{l}^{(i)}\in[\aspan(K_{l})]^{\perp}$. Then $A\gamma^{(i)}+\sum_{l=1}^{r}\tilde{w}_{l}^{(i)}$
would lie in $N_{K}(\bar{x})$ for any $A\in\mathcal{A}$ by \eqref{eq:normal-sum-formula}. 

In the case where $\gamma^{(i)}$ are zero, the numerator in \eqref{eq:key-ineq}
is zero, so things are straightforward. So we shall look only at the
subsequence for which $\gamma^{(i)}$ are nonzero. (We do not relabel.)
For the denominator, we have $\|x_{i-1}-\bar{x}\|\geq\|P_{L}(A^{(i)}\gamma^{(i)})\|$.
Then Claim 4 and \eqref{eq:A-gamma-div-gamma} imply 
\[
0\leq\lim_{i\to\infty}\left(\inf_{A\in\mathcal{A}}\frac{\|A\gamma^{(i)}-A^{(i)}\gamma^{(i)}\|}{\|P_{L}(A^{(i)}\gamma^{(i)})\|}\right)\leq\lim_{i\to\infty}\frac{\inf_{A\in\mathcal{A}}\|A-A^{(i)}\|}{\beta}=0,
\]
from which \eqref{eq:key-ineq} follows easily.\end{proof}
\begin{prop}
\label{prop:diff-btw-unit-vecs}(Intermediate estimate) Suppose $v_{1}$
and $v_{2}$ are vectors in $\mathbb{R}^{n}$ such that $\frac{\|v_{1}-v_{2}\|}{\|v_{2}\|}\leq\beta$.
Then $\left\Vert \frac{v_{1}}{\|v_{1}\|}-\frac{v_{2}}{\|v_{2}\|}\right\Vert \leq2\beta$.\end{prop}
\begin{proof}
We have 
\begin{eqnarray*}
\left\Vert \frac{v_{1}}{\|v_{1}\|}-\frac{v_{2}}{\|v_{2}\|}\right\Vert  & \leq & \left\Vert \frac{v_{1}}{\|v_{2}\|}-\frac{v_{2}}{\|v_{2}\|}\right\Vert +\left\Vert \frac{v_{1}}{\|v_{1}\|}-\frac{v_{1}}{\|v_{2}\|}\right\Vert \\
 & \leq & \beta+\frac{\|v_{1}\|\left|\|v_{1}\|-\|v_{2}\|\right|}{\|v_{1}\|\|v_{2}\|}\\
 & \leq & \beta+\frac{\|v_{1}-v_{2}\|}{\|v_{2}\|}\leq2\beta.
\end{eqnarray*}
\end{proof}
\begin{prop}
\label{prop:pointed-cone-modulus}(Pointed cone) For a closed pointed
convex cone $K\subset\mathbb{R}^{n}$, there is a unit vector $d$
in $K^{+}$, the positive polar cone of $K$, and some $c>0$ such
that $\mathbb{B}_{c}(d)\subset K^{+}$. For any unit vector $v\in K$,
we have $d^{T}v\geq c$.

Moreover, suppose $\lambda_{i}\geq0$ and $v_{i}$ are unit vectors
in $K$ for all $i$ and $\sum_{i=1}^{\infty}\lambda_{i}v_{i}$ converges
to $\bar{v}$. Clearly, $\bar{v}\in K$. Then $\|\sum_{i=1}^{\infty}\lambda_{i}v_{i}\|\geq c\sum_{i=1}^{\infty}\lambda_{i}$,
which also implies that $\sum_{i=1}^{\infty}\lambda_{i}$ is finite.\end{prop}
\begin{proof}
For the unit vector $v\in K$, we have $(d-cv)\in K^{+}$, which gives
$(d-cv)^{T}v\geq0$, from which the first part follows. 

Next, 
\[
\left\Vert \sum_{i=1}^{\infty}\lambda_{i}v_{i}\right\Vert \geq d^{T}\sum_{i=1}^{\infty}\lambda_{i}v_{i}\geq c\sum_{i=1}^{\infty}\lambda_{i},
\]
and the second part follows.\end{proof}
\begin{prop}
\label{prop:limit-est-w-pointed-cones}(Limit estimates involving
pointed cones) Suppose $\{v_{i}\}_{i}$ are unit vectors in $\mathbb{R}^{n}$
and $\{\lambda_{i}\}_{i}$ is a sequence of nonnegative numbers such
that the cluster points of $\{v_{i}\}_{i}$ belong to a closed pointed
convex cone $K\subset\mathbb{R}^{n}$. Then 
\begin{enumerate}
\item If $\sum_{i=1}^{\infty}\lambda_{i}=\infty$, then cluster points of
$\left\{ \frac{\sum_{i=1}^{j}\ensuremath{\lambda_{i}v_{i}}}{\|\sum_{i=1}^{j}\lambda_{i}v_{i}\|}\right\} _{j=1}^{\infty}$
belong to $K$.
\item Take $c>0$ to be the constant in Proposition \ref{prop:pointed-cone-modulus}.
If $\sum_{i=1}^{\infty}\lambda_{i}v_{i}$ is convergent and there
are unit vectors $\tilde{v}_{i}\in K$ such that $\|v_{i}-\tilde{v}_{i}\|\leq\epsilon$,
then 
\begin{equation}
\left\Vert \frac{\sum_{i=1}^{\infty}\lambda_{i}v_{i}}{\left\Vert \sum_{i=1}^{\infty}\lambda_{i}v_{i}\right\Vert }-\frac{\sum_{i=1}^{\infty}\lambda_{i}\tilde{v}_{i}}{\left\Vert \sum_{i=1}^{\infty}\lambda_{i}\tilde{v}_{i}\right\Vert }\right\Vert \leq\frac{2}{c}\epsilon.\label{eq:unit-normals-close}
\end{equation}

\end{enumerate}
\end{prop}
\begin{proof}
\textbf{Statement (1): }Since the cluster points of $\{v_{i}\}$ belong
to $K$, for any $\epsilon>0$, we can find $I_{\epsilon}$ such that
$\|v_{i}-\tilde{v}_{i}\|<\epsilon$ for some $\tilde{v}_{i}\in K$.
Then 
\[
\left\Vert \sum_{i=1}^{j}\lambda_{i}v_{i}-\sum_{i=1}^{j}\lambda_{i}\tilde{v}_{i}\right\Vert =\left\Vert \sum_{i=1}^{j}\lambda_{i}(v_{i}-\tilde{v}_{i})\right\Vert \leq\sum_{i=I_{\epsilon}}^{j}\lambda_{i}\epsilon+\sum_{i=1}^{I_{\epsilon}-1}2\lambda_{i}.
\]
Next, Proposition \ref{prop:pointed-cone-modulus} implies that $\left\Vert \sum_{i=1}^{j}\lambda_{i}\tilde{v}_{i}\right\Vert \geq c\sum_{i=1}^{j}\lambda_{i}.$
So 
\[
\frac{\left\Vert \sum_{i=1}^{j}\lambda_{i}v_{i}-\sum_{i=1}^{j}\lambda_{i}\tilde{v}_{i}\right\Vert }{\left\Vert \sum_{i=1}^{j}\lambda_{i}\tilde{v}_{i}\right\Vert }\leq\frac{\sum_{i=I_{\epsilon}}^{j}\lambda_{i}\epsilon+\sum_{i=1}^{I_{\epsilon}-1}2\lambda_{i}}{c\sum_{i=1}^{j}\lambda_{i}}.
\]
Proposition \ref{prop:diff-btw-unit-vecs} gives 
\[
\left\Vert \frac{\sum_{i=1}^{j}\lambda_{i}v_{i}}{\left\Vert \sum_{i=1}^{j}\lambda_{i}v_{i}\right\Vert }-\frac{\sum_{i=1}^{j}\lambda_{i}\tilde{v}_{i}}{\left\Vert \sum_{i=1}^{j}\lambda_{i}\tilde{v}_{i}\right\Vert }\right\Vert \leq2\frac{\sum_{i=I_{\epsilon}}^{j}\lambda_{i}\epsilon+\sum_{i=1}^{I_{\epsilon}-1}2\lambda_{i}}{c\sum_{i=1}^{j}\lambda_{i}}.
\]
The RHS of the above can be made arbitrarily small since $\epsilon$
can be made arbitrarily small and $j$ can be made arbitrarily big.
The term $\frac{\sum_{i=1}^{j}\lambda_{i}\tilde{v}_{i}}{\left\Vert \sum_{i=1}^{j}\lambda_{i}\tilde{v}_{i}\right\Vert }$
belongs to $K$, so Statement (1) holds.

\textbf{Statement (2):} First, since $\sum_{i=1}^{\infty}\lambda_{i}\tilde{v}_{i}$
is convergent, Proposition \ref{prop:pointed-cone-modulus} implies
\[
\left\Vert \sum_{i=1}^{\infty}\lambda_{i}\tilde{v}_{i}\right\Vert \geq c\sum_{i=1}^{\infty}\lambda_{i},
\]
 which also implies that $\sum_{i=1}^{\infty}\lambda_{i}$ is finite,
and 
\[
\left\Vert \sum_{i=1}^{\infty}\lambda_{i}v_{i}-\sum_{i=1}^{\infty}\lambda_{i}\tilde{v}_{i}\right\Vert =\left\Vert \sum_{i=1}^{\infty}\lambda_{i}(v_{i}-\tilde{v}_{i})\right\Vert \leq\sum_{i=1}^{\infty}\lambda_{i}\epsilon.
\]
Then 
\[
\frac{\left\Vert \sum_{i=1}^{\infty}\lambda_{i}v_{i}-\sum_{i=1}^{\infty}\lambda_{i}\tilde{v}_{i}\right\Vert }{\left\Vert \sum_{i=1}^{\infty}\lambda_{i}\tilde{v}_{i}\right\Vert }\leq\frac{\epsilon}{c}.
\]
By Proposition \ref{prop:diff-btw-unit-vecs}, we get the conclusion
\eqref{eq:unit-normals-close} as needed.
\end{proof}
Next, we give conditions for estimating the distance to the point
of convergence using the distance to the respective sets. We recall
the definition of local linear regularity.
\begin{defn}
\label{def:Loc-lin-reg}(Local metric inequality) We say that a collection
of closed sets $K_{l}$, $l=1,\dots,r$ satisfies the \emph{local
metric inequality }at $\bar{x}$ if there are $\beta>0$ and $\delta>0$
such that
\begin{equation}
d(x,\cap_{l=1}^{r}K_{l})\leq\beta\max_{1\leq l\leq r}d(x,K_{l})\mbox{ for all }x\in\mathbb{B}_{\delta}(\bar{x}).\label{eq:loc-metric-ineq}
\end{equation}

\end{defn}
In this paper, we shall only consider the case where $K_{l}$ are
all convex. The term linear regularity is used in two different ways
in \cite[after (15)]{Kruger_06} and \cite[Proposition 2.3]{LLM09_lin_conv_alt_proj},
so we refrain from using the term here. A concise summary of further
studies on the local metric inequality appears in \cite{Kruger_06},
who in turn referred to \cite{BBL99,Iof00,Ngai_Thera01,NgWang04}
on the topic of local metric inequality and their connection to metric
regularity. Definition \ref{def:Loc-lin-reg} is sufficient for our
purposes. The local metric inequality is useful for proving the linear
convergence of alternating projection algorithms \cite{BB93_Alt_proj,LLM09_lin_conv_alt_proj}.
See \cite{BB96_survey} for a survey. 

With the additional assumption of local metric inequality, we have
the following result.
\begin{lem}
\label{lem:iterate-estimate}(Estimates under local metric inequality)
Let $K_{l}\subset\mathbb{R}^{n}$, where $1\leq l\leq r$, be closed
convex sets. Suppose a sequence $\{x_{i}\}$ converges to the point
$\bar{x}\in K:=\cap_{l=1}^{r}K_{l}$, $\{K_{l}\}_{l=1}^{r}$ satisfies
the local metric inequality at $\bar{x}$, and
\begin{equation}
\lim_{i\to\infty}\frac{\|P_{N_{K}(\bar{x})}(x_{i}-\bar{x})\|}{\|x_{i}-\bar{x}\|}=1.\label{eq:P-N-limit}
\end{equation}
Then there is a $\beta>0$ such that 
\begin{equation}
\|x_{i}-\bar{x}\|\leq\beta\max_{1\leq l\leq r}d(x_{i},K_{l})\mbox{ for all }i\mbox{ large enough}.\label{eq:iterate-est}
\end{equation}
\end{lem}
\begin{proof}
By Moreau's Theorem, we have
\begin{eqnarray}
\|P_{T_{K}(\bar{x})}(x_{i}-\bar{x})\|^{2} & = & \|x_{i}-\bar{x}\|^{2}-\|P_{N_{K}(\bar{x})}(x_{i}-\bar{x})\|^{2}\nonumber \\
\Rightarrow\lim_{i\to\infty}\frac{\|P_{T_{K}(\bar{x})}(x_{i}-\bar{x})\|^{2}}{\|x_{i}-\bar{x}\|^{2}} & = & \lim_{i\to\infty}\left(1-\frac{\|P_{N_{K}(\bar{x})}(x_{i}-\bar{x})\|^{2}}{\|x_{i}-\bar{x}\|^{2}}\right)=0.\label{eq:P-T-limit}
\end{eqnarray}
Let $\tilde{x}_{i}$ be such that $\tilde{x}_{i}-\bar{x}=P_{N_{K}(\bar{x})}(x_{i}-\bar{x})$,
and $x_{i}-\tilde{x}_{i}=P_{T_{K}(\bar{x})}(x_{i}-\bar{x})$. Formulas
\eqref{eq:P-N-limit} and \eqref{eq:P-T-limit} give us
\begin{equation}
\lim_{i\to\infty}\frac{\|\tilde{x}_{i}-\bar{x}\|}{\|x_{i}-\bar{x}\|}=1\mbox{ and }\lim_{i\to\infty}\frac{\|\tilde{x}_{i}-x_{i}\|}{\|x_{i}-\bar{x}\|}=0.\label{eq:P-N-T-limits}
\end{equation}
Since $\tilde{x}_{i}-\bar{x}\in N_{K}(\bar{x})$, we have $d(\tilde{x}_{i},K)=\|\tilde{x}_{i}-\bar{x}\|$.
So, by the Lipschitzness of the projection operation, we have 
\begin{eqnarray}
 &  & d(\tilde{x}_{i},K)-\|\tilde{x}_{i}-x_{i}\|\leq d(x_{i},K)\leq d(\tilde{x}_{i},K)+\|\tilde{x}_{i}-x_{i}\|\nonumber \\
 & \Rightarrow & \|\tilde{x}_{i}-\bar{x}\|-\|\tilde{x}_{i}-x_{i}\|\leq d(x_{i},K)\leq\|\tilde{x}_{i}-\bar{x}\|+\|\tilde{x}_{i}-x_{i}\|.\label{eq:squeeze-d-x-i}
\end{eqnarray}
 The formulas \eqref{eq:P-N-T-limits} and \eqref{eq:squeeze-d-x-i}
give $\lim_{i\to\infty}\frac{d(x_{i},K)}{\|x_{i}-\bar{x}\|}=1$. Together
with the definition of local metric inequality \eqref{eq:loc-metric-ineq},
we can obtain what we need. 
\end{proof}
Local metric inequality follows from Condition (1) in Lemma \ref{lem:Approach-to-x-bar}.
We paraphrase the result from \cite{LLM09_lin_conv_alt_proj}, where
the authors remarked that the theorem is well known. For example,
a globalized version appears in the survey \cite[Theorem 3.7]{Bauschke01_survey}
without attribution.
\begin{lem}
\label{lem:loc-metr-ineq-condn}(Condition for local metric inequality)
Suppose $\bar{x}\in K$, where $K=\cap_{l=1}^{r}K_{l}$ and $K_{l}\subset\mathbb{R}^{n}$
for $1\leq l\leq r$, and that Condition (1) of Lemma \ref{lem:Approach-to-x-bar}
holds. Then $\{K_{l}\}_{l=1}^{r}$ satisfies the local metric inequality
at $\bar{x}$.\end{lem}
\begin{proof}
In \cite[Section 3]{LLM09_lin_conv_alt_proj}, it was proved that
if Condition (1) of Lemma \ref{lem:Approach-to-x-bar} holds, then
there is a constant $\kappa\geq0$ such that 
\[
d\Big(x,\bigcap_{i}(K_{i}-z_{i})\Big)\leq\kappa\sqrt{\sum_{i}d^{2}(x,K_{i}-z_{i})}\mbox{ for all }(x,z)\mbox{ near }(\bar{x},0),
\]
This is easily seen to be stronger than the conclusion since we only
need $z_{i}=0$ for $1\leq i\leq r$.
\end{proof}
We state the key result of this section.
\begin{thm}
\label{thm:Superlinear-conv}(Superlinear convergence) Consider the
problem of finding a point $x\in K$, where $K=\cap_{l=1}^{r}K_{l}$
and $K_{l}\subset\mathbb{R}^{n}$. Suppose Algorithm \ref{alg:Mass-proj-alg}
produces a sequence $\{x_{i}\}$ that converges to a point $\bar{x}\in K$.
Suppose also that the conditions in Lemma \ref{lem:Approach-to-x-bar}
hold, i.e., 
\begin{enumerate}
\item If $\sum_{l=1}^{r}v_{l}=0$ for some $v_{l}\in N_{K_{l}}(\bar{x})$,
then $v_{l}=0$ for all $l=1,\dots,r$.
\end{enumerate}
If $\bar{p}$ in Algorithm \ref{alg:Mass-proj-alg} is sufficiently
large, then we have 
\begin{equation}
\limsup_{i\to\infty}\frac{\|x_{i+\bar{p}}-\bar{x}\|}{\|x_{i}-\bar{x}\|}=0.\label{eq:superlin-conv}
\end{equation}
Moreover, for that choice of $\bar{p}$, if 
\begin{equation}
\mbox{for some }\bar{\epsilon}>0\mbox{, }[K_{l}-\bar{x}]\cap\bar{\epsilon}\mathbb{B}=T_{K_{l}}(\bar{x})\cap\bar{\epsilon}\mathbb{B}\mbox{ for all }l=1,\dots,r,\label{eq:cone-condn-1}
\end{equation}
 then the convergence of $\{x_{i}\}$ to $\bar{x}$ is finite.\end{thm}
\begin{proof}
In Algorithm \ref{alg:Mass-proj-alg}, let $l_{i}\in\{1,\dots,r\}$
be such that 
\[
l_{i}\in\arg\max_{1\leq l\leq r}\|x_{i}-P_{K_{l}}(x_{i})\|=\arg\max_{1\leq l\leq r}d(x_{i},K_{l}).
\]
Let $v_{i}^{*}$ be the unit vector $v_{i}^{*}:=\frac{x_{i}-P_{K_{l_{i}}}(x_{i})}{\|x_{i}-P_{K_{l_{i}}}(x_{i})\|}$.
In other words, $v_{i}^{*}$ is the unit vector of the hyperplane
that separates $x_{i}$ from $K_{l_{i}}$. 

Without loss of generality, suppose that $\bar{x}=0$. Suppose $\beta>0$
is chosen such that \eqref{eq:iterate-est} holds. From Lemma \ref{lem:loc-metr-ineq-condn},
we deduce that $\{K_{l}\}_{l=1}^{r}$ satisfies the local metric inequality
at $\bar{x}$.

The sphere $S^{n-1}:=\{w\in\mathbb{R}^{n}\mid\|w\|=1\}$ is compact.
Suppose $\bar{p}$ is such that we can cover $S^{n-1}$ with $\bar{p}$
balls of radius $\frac{1}{4\beta}$. 

Next, among the vectors $\{v_{i}^{*},v_{i+2}^{*},\dots,v_{i+\bar{p}}^{*}\}$,
there must exist $j$ and $k$ such that $i\leq j<k\leq i+\bar{p}$,
and $v_{j}^{*}$ and $v_{k}^{*}$ belong to the same ball of radius
$\frac{1}{4\beta}$ covering $S^{n-1}$. We thus have $\|v_{j}^{*}-v_{k}^{*}\|\leq\frac{1}{2\beta}$.
We can assume, using Theorem \ref{thm:radiality}, that $i$ is large
enough so that 
\begin{equation}
v_{j}^{*T}x_{k}\leq\epsilon\|x_{j}\|.\label{eq:epsilon-x1}
\end{equation}
On the other hand, if $i$ is large enough, we can apply Lemma \ref{lem:iterate-estimate}
to get
\begin{eqnarray}
v_{j}^{*T}x_{k} & = & v_{k}^{*T}x_{k}+(v_{j}^{*}-v_{k}^{*})^{T}x_{k}\nonumber \\
 & \geq & d(x_{k,}K_{l_{k}})-\frac{1}{2\beta}\|x_{k}\|\nonumber \\
 & \geq & \frac{1}{2\beta}\|x_{k}\|.\label{eq:beta-x2}
\end{eqnarray}
The methods in Theorem \ref{thm:basic-conv} can be easily adapted
to prove that the sequence $\{x_{i}\}$ is Fej\'{e}r monotone with
respect to $K$. The inequalities \eqref{eq:epsilon-x1} and \eqref{eq:beta-x2},
and the Fej\'{e}r monotonicity of $\{x_{i}\}$ combine to give 
\[
\|x_{i+\bar{p}}\|\leq\|x_{k}\|\leq2\beta\epsilon\|x_{j}\|\leq2\beta\epsilon\|x_{i}\|.
\]
As the factor $\epsilon$ can be made arbitrarily close to $0$, we
proved \eqref{eq:superlin-conv}.

Next, under the added condition \eqref{eq:cone-condn-1}, the formula
\eqref{eq:epsilon-x1} becomes $v_{j}^{*T}x_{k}\leq0$ instead by
an application of Moreau's Theorem (See Proposition \ref{rem:Use-Moreau}),
and the same steps show us that $\frac{1}{2\beta}\|x_{k}\|\leq0$,
which forces $x_{k}=0$, or $x_{k}=\bar{x}$. 
\end{proof}
Even though the choice of $\bar{p}$ in the proof of Theorem \ref{thm:Superlinear-conv}
is impractical, Theorem \ref{thm:Superlinear-conv} gives justification
that the idea of supporting hyperplanes and quadratic programming
can lead to fast convergence.

\subsection{Alternative estimates}

We close this section with a result that might be helpful for estimating
the distance of an iterate to the limit $\bar{x}$. 
\begin{lem}
\label{lem:conv-est}(Alternative estimate) Let $K:=\cap_{l=1}^{r}K_{l}$,
where $K_{l}$ are closed convex sets in $\mathbb{R}^{n}$ for $1\leq l\leq r$.
Let hyperplanes $H_{j}:=\{x\mid\left\langle a_{j},x\right\rangle =b_{j}\}$,
points $a_{j}\in\mathbb{R}^{n}$ and $\tilde{x}_{j}\in\mathbb{R}^{n}$,
where $j=1,\dots,J$, be such that $\|a_{j}\|=1$. Suppose $x^{*}\in\mathbb{R}^{n}$
lies on the hyperplanes $H_{j}$. Let $\bar{x}\in K$ be such that 
\begin{enumerate}
\item Each hyperplane $H_{j}$ is a supporting hyperplane to some $K_{l_{j}}$
and $K_{l_{j}}\subset\{x\mid\left\langle a_{j},x\right\rangle \leq b_{j}\}$. 
\item $\tilde{x}_{j}\in H_{j}\cap K_{l_{j}}$, 
\item $\max_{j}\|\tilde{x}_{j}-\bar{x}\|=L$,
\item There is some $\epsilon>0$ such that $-\epsilon\leq\frac{\left\langle a_{j},\bar{x}-\tilde{x}_{j}\right\rangle }{\|\bar{x}-\tilde{x}_{j}\|}\le0$
for all $j=1,\dots,J$.
\end{enumerate}
Let $\underline{\sigma}$ be the smallest singular value of the matrix
$A\in\mathbb{R}^{n\times J}$, where the $j$th column of $A$ is
$a_{j}$. Let $S$ be $\mbox{span}\{a_{1},\dots,a_{J}\}$. Let $\alpha$
be such that 
\begin{equation}
\|M\|_{\infty,2}\leq\alpha\|M\|_{2,2}\mbox{ for all }M\in\mathbb{R}^{n\times J},\label{eq:matrix-norm-ratio}
\end{equation}
where $\|M\|_{p,q}:=\sup_{v\neq0}\frac{\|Mv\|_{q}}{\|v\|_{p}}$. Then
\begin{equation}
\|P_{S}(x^{*}-\bar{x})\|\leq L\epsilon\alpha\underline{\sigma}^{-1}.\label{eq:proj-less-than-orig-bdd}
\end{equation}
\end{lem}
\begin{proof}
Since $x^{*}$ is in $H_{j}$, we have $\left\langle a_{j},x^{*}\right\rangle =b_{j}$.
By Conditions (3) and (4), we get 
\[
-\epsilon L\leq-\epsilon\|\bar{x}-\tilde{x}_{j}\|\leq\left\langle a_{j},\bar{x}-\tilde{x}_{j}\right\rangle \leq0.
\]
Since $\left\langle a_{j},\tilde{x}_{j}\right\rangle =b_{j}=\left\langle a_{j},x^{*}\right\rangle $,
we have 
\begin{equation}
0\leq\left\langle a_{j},x^{*}-\bar{x}\right\rangle \leq\epsilon L.\label{eq:a-T-x-estimate}
\end{equation}
By standard linear least squares, we have 
\begin{eqnarray*}
\|P_{S}(x^{*}-\bar{x})\|_{2} & = & \|A(A^{T}A)^{-1}A^{T}(x^{*}-\bar{x})\|_{2}\\
 & \leq & \|A(A^{T}A)^{-1}\|_{\infty,2}\|A^{T}(x^{*}-\bar{x})\|_{\infty}
\end{eqnarray*}
By \eqref{eq:a-T-x-estimate}, we have $\|A^{T}(x^{*}-\bar{x})\|_{\infty}\leq\epsilon L$.
Furthermore, using standard properties of the singular value decomposition,
we have 
\[
\|A(A^{T}A)^{-1}\|_{\infty,2}\leq\alpha\|A(A^{T}A)^{-1}\|_{2,2}=\alpha\underline{\sigma}^{-1}.
\]
The required bound follows immediately.
\end{proof}
 To apply Lemma \ref{lem:conv-est} to Algorithm \ref{alg:Mass-proj-alg},
note that Condition (3) follows from properties of the projection,
while Condition (4) is an attempt to apply Theorem \ref{thm:radiality}.
Lemma \ref{lem:conv-est} is closer to the spirit of Theorem \ref{thm:Superlin}.
However, the term $\underline{\sigma}^{-1}$ is hard to control, so
we have not had success in applying Lemma \ref{lem:conv-est} so far.

\section{\label{sec:Infeasibility}Infeasibility}

 We now discuss the case where the $K:=\cap_{l=1}^{r}K_{l}=\emptyset$.
For any algorithm producing a sequence $\{x_{i}\}$ in the hope of
converging to a limit $\bar{x}\in K$, there are three possibilities:
\begin{enumerate}
\item An infinite sequence cannot be produced because the intersection of
the halfspaces is an empty set at some point.
\item The sequence $\{x_{i}\}$ contains a cluster point $\bar{x}$.
\item The sequence $\{x_{i}\}$ does not contain a cluster point $\bar{x}$. 
\end{enumerate}
We first show that case 2 is not possible for Algorithm \ref{alg:closest-pt-proj}
in the case of strong cluster points.
\begin{thm}
\label{thm:No-cluster-pt}(No cluster point) For Algorithm \ref{alg:closest-pt-proj}
using \eqref{eq:mass-proj}, in the case where $K=\emptyset$, the
sequence $\{x_{i}\}$ cannot contain a strong cluster point.\end{thm}
\begin{proof}
Suppose on the contrary that $\{x_{i}\}$ contains a strong cluster
point, say $\tilde{x}$. Since $\tilde{x}\notin K$, we assume without
loss of generality that $\tilde{x}\notin K_{1}$. Then let $z:=P_{K_{1}}(\tilde{x})$,
and $v=\tilde{x}-z$. Let $a_{x}=x-P_{K_{1}}(x)$ and $b_{x}=\left\langle a_{x},P_{K_{1}}(x)\right\rangle $.
By elementary properties of the projection, we have $\left\langle a_{\tilde{x}},\tilde{x}\right\rangle >b_{\tilde{x}}$.
The parameters $a_{x}$ and $b_{x}$ depend continuously on $x$.
By the workings of Algorithm \ref{alg:closest-pt-proj}, we have 
\[
\left\langle a_{x_{i}},x_{i+1}\right\rangle \leq b_{x_{i}}.
\]
As we take limits as $i\to\infty$, we get $\left\langle a_{\tilde{x}},\tilde{x}\right\rangle \leq b_{\tilde{x}}$.
This is a contradiction.
\end{proof}
One can easily check that Case 3 can happen. Consider the sets $K_{1}$
and $K_{2}$ defined by 
\begin{eqnarray*}
K_{1} & = & \{(x,y)\in\mathbb{R}^{2}\mid y\geq e^{-x}\},\\
\mbox{ and }K_{2} & = & \{(x,y)\in\mathbb{R}^{2}\mid y\leq-e^{-x}\}.
\end{eqnarray*}
If $x_{0}$ is chosen to be the origin in Algorithm \ref{alg:closest-pt-proj},
then the iterates $x_{i}$ cannot converge to a limit by Theorem \ref{thm:No-cluster-pt},
and therefore must move in the direction of the positive $x$ axis.
We understand more about such behavior with the result below.
\begin{thm}
(Recession directions) If $\{x_{i}\}$ is a sequence of iterates for
Algorithm \ref{alg:closest-pt-proj} using \eqref{eq:mass-proj} in
the case where $K=\emptyset$ and $X=\mathbb{R}^{n}$, then any cluster
point of $\{\frac{x_{i}}{\|x_{i}\|}\}$ must lie in $R(K_{l})$, the
recession cone of $K_{l}$, for all $l=1,\dots,r$.\end{thm}
\begin{proof}
Let $\{\frac{\tilde{x}_{i}}{\|\tilde{x}_{i}\|}\}$ be a subsequence
of $\{\frac{x_{i}}{\|x_{i}\|}\}$ which has a limit $v$. We show
that such a limit has to lie in $R(K_{l})$. Seeking a contradiction,
suppose that $v\notin R(K_{l})$. 

We show that there is a unit vector $w\in\mathbb{R}^{n}$ and $M\in\mathbb{R}$
such that $w^{T}c\leq M$ for all $c\in K_{l}$ and $w^{T}v>0$. Take
any point $y\in K_{l}$. Since $v\notin R(K_{l})$, there is some
$\gamma\geq0$ such that $y+\gamma v\in K_{l}$, but $y+\gamma^{\prime}v\notin K_{l}$
for all $\gamma^{\prime}>\gamma$. It follows that there exists a
unit vector $w\in N_{K_{l}}(y+\gamma v)$ such that $w^{T}v>0$, and
we can take $M=w^{T}(y+\gamma v)$. Since $w^{T}v>0$, we shall assume
that $w^{T}\tilde{x}_{i}>M$ for all $i$.

Let $c_{i}:=P_{K_{l}}(\tilde{x}_{i})$, and let $u_{i}$ be the unit
vector in the direction of $\tilde{x}_{i}-c_{i}$. We write $\tilde{x}_{i}-c_{i}=\alpha_{i}u_{i}$.
We have 
\begin{eqnarray*}
u_{i}^{T}c_{i} & = & u_{i}^{T}(\tilde{x}_{i}-\alpha_{i}u_{i})\\
 & = & u_{i}^{T}\tilde{x}_{i}-\alpha_{i}.
\end{eqnarray*}
Also
\begin{eqnarray*}
\alpha_{i}w^{T}u_{i} & = & w^{T}\tilde{x}_{i}-w^{T}c_{i}\\
 & \geq & w^{T}\tilde{x}_{i}-M.
\end{eqnarray*}
Since $\alpha_{i}w^{T}u_{i}=w^{T}(\tilde{x}_{i}-c_{i})>M-M=0$, we
have $w^{T}u_{i}>0$, and hence $\alpha_{i}\geq w^{T}\tilde{x}_{i}-M$.
Therefore, 
\[
u_{i}^{T}c_{i}\leq u_{i}^{T}\tilde{x}_{i}-w^{T}\tilde{x}_{i}+M.
\]
By the workings of Algorithm \ref{alg:closest-pt-proj}, we have $u_{i}^{T}\tilde{x}_{i}>u_{i}^{T}c_{i}$
and $u_{i}^{T}\tilde{x}_{j}\leq u_{i}^{T}c_{i}$ for all $j>i$. This
gives $u_{i}^{T}(\tilde{x}_{j}-\tilde{x}_{i})\leq0$, which gives
$u_{i}^{T}v\leq0$. 

Let $u$ be a cluster point of $\{u_{i}\}$. We can consider subsequences
so that $\lim_{i\to\infty}u_{i}$ exists. For any point $c\in K_{l}$,
we have 
\begin{eqnarray*}
u^{T}c & = & \lim_{i\to\infty}u_{i}^{T}c\\
 & \leq & \liminf_{i\to\infty}u_{i}^{T}c_{i}\\
 & \leq & \liminf_{i\to\infty}[u_{i}^{T}\tilde{x}_{i}-w^{T}\tilde{x}_{i}+M]\\
 & = & \liminf_{i\to\infty}\|\tilde{x}_{i}\|\left(u_{i}^{T}\frac{\tilde{x}_{i}}{\|\tilde{x}_{i}\|}-w^{T}\frac{\tilde{x}_{i}}{\|\tilde{x}_{i}\|}\right)+M\\
 & = & \liminf_{i\to\infty}\|\tilde{x}_{i}\|[u_{i}^{T}v-w^{T}v]+M\\
 & = & -\infty,
\end{eqnarray*}
which is absurd. The contradiction gives $v\in R(K_{l})$. 
\end{proof}

\section{Conclusion}

In this paper, we focus on the theoretical properties of using supporting
hyperplanes and quadratic programming to accelerate the method of
alternating projections and its variants. It appears that as long
as a separating hyperplane is obtained for $K$ and the quadratic
programs are not too big, it is a good idea to solve the associated
quadratic program to obtain better iterates. Other issues to consider
in a practical implementation would be to either remove or combine
loose constraints so that the size of the intermediate quadratic programs
do not get too big. The ideas in \cite{Kiwiel95} for example can
be useful. It remains to be seen whether the theoretical properties
in this paper translate to effective algorithms in practice.

\section{Acknowledgments}

We thank Boris Mordukhovich for asking the question of whether the
alternating projection algorithm can achieve superlinear convergence
at the ISMP in 2012 during Russell Luke's talk, which led to the idea
of considering supporting hyperplanes.

\bibliographystyle{amsalpha}
\bibliography{refs}

\end{document}